\documentclass[a4paper]{amsart}
\usepackage{srcltx} 
\usepackage{amsmath,amsthm}
\usepackage{mathrsfs}
\usepackage{amssymb}
\usepackage{enumitem}

\newtheorem{theor}{Theorem}[section]
\newtheorem{cor}[theor]{Corollary}
\newtheorem{lemma}[theor]{Lemma}
\newtheorem{prop}[theor]{Proposition}
\newtheorem{property}[theor]{Property}

\theoremstyle{definition}
\newtheorem{defn}{Definition}

\theoremstyle{remark}
\newtheorem{step}{Step}
\newtheorem{case}{Case}
\newtheorem{remark}[theor]{Remark}

\numberwithin{equation}{section}
\numberwithin{defn}{section}

\newcommand{\N}{\mathbb{N}}        
\newcommand{\R}{\mathbb{R}}        
\newcommand{\C}{\mathbb{C}}        
\newcommand{\PR}{\mathbb P^2(\mathbb R)}

\DeclareMathOperator{\Id}{Id}
\DeclareMathOperator{\tr}{tr}
\DeclareMathOperator{\dist}{dist}
\renewcommand{\d}{\mathrm{d}}
\newcommand{\abs}[1]{\left| #1 \right|}
\newcommand{\norm}[1]{\left\| #1 \right\|}

\newcommand{\NN}{\mathscr N}
\newcommand{\Sz}{\mathbf{S}_0}
\newcommand{\Qe }{Q_\varepsilon}
\newcommand{\Pe}{P_\varepsilon}
\newcommand{\ue}{u_\varepsilon}
\newcommand{\un}{u_{\varepsilon_n}}
\newcommand{\vn}{v_{\varepsilon_n}}
\newcommand{\dn}{\sigma_{\varepsilon_n}}
\newcommand{\ee}{e_\varepsilon}

\usepackage[usenames,dvipsnames]{pstricks}
\usepackage{epsfig}
\usepackage{pst-grad} 
\usepackage{pst-plot} 

\begin{document}
\title[Biaxiality in a 2D Landau-de Gennes model]{Biaxiality in the asymptotic analysis of a 2D Landau-de Gennes model for liquid crystals} 
\author{Giacomo Canevari}
\address{Sorbonne Universit\'es,
  UPMC --- Universit\'e Paris 06,
  UMR 7598, Laboratoire Jacques-Louis Lions,
  \mbox{F-75005}, Paris, France,
  {\tt canevari@ann.jussieu.fr}}
%
\date{December 12, 2013}
\begin{abstract} 
We consider the Landau-de Gennes variational problem on a bound\-ed, two dimensional domain, subject to Dirichlet smooth boundary conditions.
We prove that minimizers are maximally biaxial near the singularities, that is, their biaxiality parameter reaches the maximum value $1$.
Moreover, we discuss the convergence of minimizers in the vanishing elastic constant limit. Our asymptotic analysis is performed in a general setting, which 
recovers the Landau-de Gennes problem as a specific case.
 \end{abstract}
%
%
\subjclass[2010]{35J25, 35J61, 35B40, 35Q70.} 
\keywords{Landau-de Gennes model, $Q$-tensor, convergence, biaxiality.}
\maketitle

\section{Introduction}

Nematic liquid crystals are an intermediate phase of matter, which shares some properties both with solid and liquid states.
They are composed by rigid, rod-shaped molecules which can flow freely, as in a conventional liquid, but tend to align locally along some directions, thus recovering,
to some extent, long-range orientational order. As a result, liquid crystals behave mostly like fluids, but exhibit anisotropies with respect to some optical or
electromagnetic properties, which makes them suitable for many applications.

In the mathematical and physical literature about liquid crystals, different continuum theories have been proposed. 
Some of them --- like the Oseen-Frank and the Ericksen theories --- postulate that, at every point, the locally preferred direction of molecular alignment is unique: 
such a behavior is commonly referred to as \emph{uniaxiality}, and materials which exhibit such a property are said to be in the uniaxial phase. 
In contrast, the Landau-de Gennes theory, which is considered here, allows \emph{biaxiality}, that is, more than one preferred direction of molecular orientation might coexist at some point. There is experimental evidence for the existence of thermotropic biaxial phases, that is, biaxial phases whose transitions are induced by temperature (see \cite{Madsen, Prasad}).

In the Landau-de Gennes theory (or, as it is sometimes informally called, the $Q$-tensor theory), the local configuration of the liquid crystal is modeled with a real $3 \times 3$ symmetric traceless matrix $Q(x)$, depending on the position $x$.
The configurations are classified according to the eigenvalues of $Q$. More precisely, $Q = 0$ corresponds to an isotropic phase (i.e., completely lacking of orientational order), matrices $Q \neq 0$ with two identical eigenvalues represent uniaxial phases, and matrices whose eigenvalues are pairwise distinct describe biaxial phases.
Every $Q$-tensor can be represented as follows:
\begin{equation} \label{intro repr}
 Q = s \left\{ \left(n^{\otimes 2} - \frac13\Id\right) + r \left(m^{\otimes 2} - \frac13\Id\right)\right\}
\end{equation}
with $0\leq r \leq 1$, $s\geq 0$ and $(n, \, m)$ is a positively oriented orthonormal pair in $\R^3$. The parameters $s$ and $r$ are respectively related to the modulus and the biaxiality of $Q$ (in particular, $Q$ is uniaxial if and only if $r\in \{0, \, 1\}$).
 
Here, we consider a two-dimensional model. The material is contained in a bounded, smooth domain $\Omega\subseteq\R^2$, subject to smooth Dirichlet boundary conditions. The configuration parameter $Q$ is assumed to minimize the Landau-de Gennes energy functional, which can be written, in its simplest form, as
\begin{equation} \label{energy 1}
 E_\varepsilon(Q) = \int_\Omega \left\{ \frac12 \abs{\nabla Q}^2 + \frac{1}{\varepsilon^2} f(Q) \right\} .
\end{equation}
Here, $\abs{\nabla Q}^2 = \sum_{i, j, k} Q_{ij, \, k}^2$ is a term penalizing the inhomogeneities in space, and $f$ is the bulk potential, given by
\begin{equation} \label{energy 2}
 f(Q) = \frac{\alpha(T - T^*)}{2}\tr Q^2 - b\tr Q^3 + c\left(\tr Q^2\right)^2 .
\end{equation}
The parameters $\alpha, \, b$ and $c$ depend on the material, $T$ is the absolute temperature, which we assume to be constant, and $T^*$ is a characteristic temperature of the liquid crystal.
We work here in the low temperature regime, that is, $T < T^*$. It can be proved (see \cite[Proposition 9]{BM}) that $f$ attains its minimum on a manifold $\NN$, termed the vacuum manifold, whose elements are exactly the matrices having $s = s_*$, $r = 0$ in the representation formula \eqref{intro repr} ($s_*$ is a parameter depending only on $\alpha, \, b, \, c, \, T$).
The potential energy $\varepsilon^{-2}f(Q)$ can be regarded as a penalization term, associated to the constraint $Q\in \NN$. In particular, as we will explain further on, biaxiality is penalized.
The parameter $\varepsilon^2$ is a material-dependent elastic constant, which is typically very small (of the order of $10^{-11}$ Jm$^{-1}$): this motivates our interest in the limit as $\varepsilon \searrow 0$.

\begin{figure}[t] \label{cone}
\begin{center}
\includegraphics[width =.43\textwidth, keepaspectratio = true]{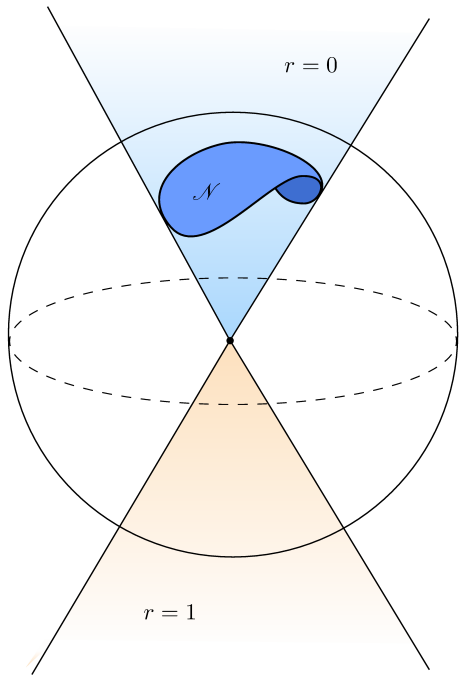}
\end{center} 
\caption{The space of $Q$-tensors. The unit sphere and the uniaxial cones, corresponding to $r = 0$ and $r = 1$, are represented. The vacuum manifold is the intersection between the sphere and the cone $r = 0$.}
\end{figure}

Due to the form of the functional \eqref{energy 1}, there are some similarities between this problem and the Ginzburg-Landau model for superconductivity, where the configuration space is the complex field $\C \simeq \R^2$, the energy is given by
\[
 E_\varepsilon(u) = \int_\Omega \left\{ \frac12 \abs{\nabla u}^2 + \frac{1}{4\varepsilon^2} \left(1 - \abs{u}^2\right)^2 \right\}
\]
and the vacuum manifold is the unit circle. The convergence analysis for this model is a widely addressed issue in the literature (see, for instance, \cite{BBH} for the study of the 2D case).
A well-known phenomenon is the appearance of the so-called topological defects. Depending on the homotopic properties of the boundary datum, there might be an obstruction to the existence of smooth maps $\Omega \to\NN$.
Boundary data for which this obstruction occurs will be referred to as homotopically non trivial (see Subsection~\ref{subsect: cost} for a precise definition). In this case, the image of minimizers fails to lie close to the vacuum manifold on some small set which correspond, in the limit as $\varepsilon\searrow 0$, to the singularities of the limit map.

In the Ginzburg-Landau model, the whole configuration space $\C$ can be recovered as a topological cone over the vacuum manifold. In other words, every configuration $u\in\C\setminus\{0\}$ is identified by its modulus and phase, the latter being associated with an element of the vacuum manifold.
Defects are characterized as the regions where $\abs{u}$ is small. This structure is found in other models: for instance, let us mention the contibution of D. Chiron (\cite{Ch}), who replaced~$\C$ by a cone over a generic compact, connected manifold. 

In contrast, this property is lost in the Landau-de Gennes model \eqref{energy 1}--\eqref{energy 2}. As a result, for the minimizers $\Qe$ of the Landau-de Gennes functional several behaviors near the singularities are possible. 
For instance, one might ask whether the image of $\Qe$ lies entirely in the cone over the vacuum manifold or not. In view of the representation formula \eqref{intro repr}, these alternatives correspond, respectively, to uniaxiality and biaxiality.

Numerical simulations suggest that we might expect biaxiality in the core of singularities. Schopohl and Sluckin (see \cite{SS}) claimed that the core is heavily biaxial at all temperatures, and that it does not contain isotropic liquid.
In the 3D case, a special biaxial configuration, known as ``biaxial torus'', has been identified in the core of point defects (see \cite{GM, KV, KVZ, SKH}).
Gartland and Mkaddem (\cite{GM}) proved  that, when $\Omega = B(0, \, R)\subseteq \R^3$ with $R$ large enough, and the boundary data is radially symmetric, the radially symmetric uniaxial configurations become unstable for sufficiently low temperature, hence minimizers cannot be purely uniaxial.
Similar conclusions have been drawn by Henao and Majumdar (\cite{HM}), by Ignat et al. (\cite{INSZ}) and, in the 2D case, by Lamy (\cite{Lamy}). 

However, in all these works, a function is said to be ``purely uniaxial'' when the parameter $r$ in \eqref{intro repr} is identically equal either to $0$ or $1$. 
Therefore, these results do not exclude the existence of an ``almost uniaxial'' minimizer, for which $r$ is very close to zero but vanishes nowhere (except, possibly, on a negligible set).

To overcome this issue, the notion of maximally biaxial configuration is introduced in Subsection~\ref{standard LG}.
One could define it as a configuration for which $r(\Qe(x)) = 1/2$ holds, at some point~$x\in\Omega$. The value $1/2$, being equidistant from $0$ and $1$, might be thought as the maximum degree of biaxiality. We are able to prove that minimizers are maximally biaxial, 
in the low temperature regime $T \ll T_*$. More precisely, we have the following

\begin{theor} \label{thm: intro-biaxial}
Assume that the boundary datum is a smooth function $g\colon \partial\Omega \to \NN$ and is not homotopically trivial. Then, there exist $t_0 > 0$ and $\varepsilon_0 = \varepsilon_0(\alpha, \, b, \, c, \, T_* - T)$ such that the two conditions
\[
 \frac{\alpha c(T_* - T)}{b^2} \geq t_0 \qquad \textrm{and} \qquad \varepsilon \leq \varepsilon_0
\]
imply that any minimizer $Q_\varepsilon$ of the Landau-de Gennes model \eqref{energy 1}--\eqref{energy 2} is maximally biaxial and satisfies
\begin{equation} \label{no isotropy}
 \inf_{\overline\Omega} \abs{Q_\varepsilon} > 0 .
\end{equation}
\end{theor}

Theorem~\ref{thm: intro-biaxial} prevents the isotropic phases ($Q = 0$) from appearing in minimizers, at the low temperature regime. This is a remarkable difference between the Landau-de Gennes theory and the popular Ericksen model for liquid crystals: in the latter, defects are always associated with isotropic melting, since biaxiality is not taken into account. Remark that Theorem~\ref{thm: intro-biaxial} is in agreement with the conclusions of \cite{SS}.

The proof of this result relies on energy estimates. With the help of the coarea formula, we are able to bound from below the energy of any uniaxial configuration. Then, we provide an explicit example of maximally biaxial solution, whose energy is smaller than the bound we have obtained, and we conclude that uniaxial minimizers cannot exist.

Another topic we discuss in this paper is the convergence of minimizers as $\varepsilon\searrow 0$. It turns out that a convergence result for the minimizers of \eqref{energy 1}--\eqref{energy 2} can be established without any need to exploit the matricial structure of the configuration space, nor the precise shape of $f$ and $\NN$.
For this reason, we introduce a more general problem, where the set of matrices is replaced by the Euclidean space $\R^d$, $\NN$ is any compact, connected submanifold of $\R^d$, and $f\colon \R^d \to [0, +\infty)$ is a smooth function, vanishing on $\NN$, which satisfies the assumptions \ref{hp: f}--\ref{hp: N} listed in Section~\ref{sect: assumptions}.
To avoid confusion, we denote by $u\colon \Omega \to \R^d$ the unknown for the new problem, and we let $\ue$ be a minimizer.

\begin{prop} \label{prop: intro-defects}
Assume that conditions \ref{hp: f}--\ref{hp: N} hold. There exist some $\varepsilon$-independent constants $\lambda_0, \, \delta_0 > 0$  and, for each $\delta \in (0, \, \delta_0)$, a finite set $X_\varepsilon = X_\varepsilon(\delta) \subset \Omega$, whose cardinality is bounded independently of $\varepsilon$, such that
 \[
  \dist(x, \, X_\varepsilon) \geq \lambda_0 \varepsilon \qquad \textrm{implies that} \qquad \dist(u_\varepsilon(x), \, \NN) \leq \delta .
 \]
\end{prop}

The set $X_\varepsilon$ is empty if and only if the boundary datum is homotopically trivial. In the Landau-de Gennes case \eqref{energy 1}--\eqref{energy 2}, Proposition~\ref{prop: intro-defects} and Theorem~\ref{thm: intro-biaxial}, combined, show that a minimizer $\Qe$ is ``almost uniaxial'' everywhere, except on $k$ balls of radius comparable to~$\varepsilon$, where biaxiality occurs. Actually, we will prove that $k = 1$ (see Proposition~\ref{prop: intro-geodetica}).

We can show that the minimizers converge, as $\varepsilon\searrow 0$, to a map taking values in $\NN$, having a finite number of singularities. Moreover, due to the variational structure of the problem, the limit map is optimal, in some sense, with respect to the Dirichlet integral $v \mapsto \frac12 \int_\Omega \abs{\nabla v}^2$. 

\begin{theor} \label{th: intro-convergence}
 Under the assumptions \ref{hp: f}--\ref{hp: N}, there exists a subsequence of $\varepsilon_n \searrow 0$, a finite set $X \subset \Omega$ and a function $u_0\in C^\infty( \Omega \setminus X, \, \NN)$
 such that 
 \[
  \un \to u_0  \qquad \textrm{strongly in } H^1_{\textrm{loc}} \cap C^0 (\Omega\setminus X, \, \R^d) .
 \]
 On every ball $B \subset\subset \Omega\setminus X$, the function $u_0$ is minimizing harmonic, which means
 \[
  \frac12 \int_B \abs{\nabla u_0}^2 = \min \left\{ \frac12 \int_B \abs{\nabla v}^2 \colon v\in H^1(B, \, \NN), \: v = u_0 \textrm{ on } \partial B  \right\} .
 \]
\end{theor}

In particular, $u_0$ is a solution of the harmonic map equation
\[
 \Delta u_0 (x) \perp T_{u_0(x)} \NN \qquad \textrm{for all } x\in\Omega\setminus X ,
\]
where $T_{u_0(x)} \NN$ is the tangent plane of $\NN$ at the point $u_0(x)$ and the symbol $\perp$ denotes orthogonality.

We can provide some information about the behavior of $u_0$ around the singularity. For the sake of simplicity, we assume here that $\NN$ is the real projective plane $\PR$ (this is the case, for instance, of the Landau-de Gennes potential \eqref{energy 2}); however, the analytic tools we employ carry over to a general manifold.

\begin{prop} \label{prop: intro-geodetica}
 In addition to \ref{hp: f}--\ref{hp: g}, assume $\NN \simeq \PR$ and the boundary datum is not homotopically trivial (see Definition~\ref{def: trivial}). Then, $X$ reduces to a singleton $\{a\}$. For $\rho\in (0, \, \dist(a, \, \partial\Omega))$, consider the function $S^1 \to \NN$ given by
 \[
  c_\rho \colon \theta \mapsto u_0\left(a + \rho e^{i\theta}\right) .
 \]
 Up to a subsequence $\rho_n\searrow 0$, $c_{\rho_n}$ converges uniformly (and in $C^{0, \, \alpha}$ for $\alpha < 1/2$) to a geodesic~$c_0$ in $\NN$, which minimizes the length among the non homotopically trivial loops in $\NN$.
\end{prop}

Unfortunately, we have not been able to prove the convergence for the whole family $(c_\rho)_{\rho > 0}$, which remains still an open question.

A interesting question, related to the topics we discuss in this paper, is the study of the singularity profile for defects in the Landau-de Gennes model. Consider a singular point $a\in X$, and set $P_\varepsilon (x) := Q_\varepsilon(a + \varepsilon x)$ for all $x\in\R^2$ for which this expression is well-defined. Then $P_\varepsilon$ is a bounded family in $L^\infty$ (see Lemma~\ref{lemma: L infty}) and it is clear, by scaling arguments, that
\[
 \norm{\nabla P_\varepsilon}_{L^2(K)} \leq C \qquad \textrm{for all } K \subset\subset\R^2 .
\]
Thus, up to a subsequence, $P_\varepsilon$ converges weakly in $H^1_{\textrm{loc}}(\R^2)$ to some $P_*$. It is readily seen that, for each $R > 0$, $P_*$ minimizes in $B(0, \, R)$ the functional $E_1$ among the functions $P\in H^1(B(0, \, R))$ satisfying $P = P_*$ on $\partial B(0, \, R)$, and consequently it solves in $\R^2$ the Euler-Lagrange equation associated with $E_1$.

A function $P_*$ obtained by this construction is called a singularity profile. 
Understanding the properties of such a profile will lead to a deeper comprehension of what happens in the core of defects, and vice-versa. Remark that, in view of Theorem~\ref{thm: intro-biaxial}, strong biaxiality has to be found in singularity profiles correpsonding to low temperatures.
We believe that the study of these objects will also play an important role in the analysis of the three-dimensional problem. Let us mention here that some results in this direction have been obtained by Henao and Majumdar, in \cite{HM}, where a 3D problem with radial symmetry is considered.
Restricting the problem to the class of uniaxial $Q$-tensors, the authors proved convergence to a radial hedgehog profile.

This paper is organized as follows. In Section~\ref{sect: assumptions} we present in detail our general problem, we set notations, and we introduce some tools for the subsequent analysis.
More precisely, in Subsection~\ref{subsect: cost} we define the energy cost of a defect, while we discuss in Subsection~\ref{subsect: geometric lemma} the nearest point projection on a manifold.
Section~\ref{sect: biaxial} specifically pertains to the $Q$-tensor model, and contains the proof of Theorem~\ref{thm: intro-biaxial}. The asymptotic analysis, with the proof of Proposition~\ref{prop: intro-defects} and Theorem~\ref{th: intro-convergence}, is provided in Section~\ref{sect: convergence}.
Finally, Section~\ref{sect: singularity} deals with Proposition~\ref{prop: intro-geodetica}.

\textbf{Note added in proof.} While preparing this paper, we were informed that Golovaty and Montero (\cite{GolMon}) have recently obtained similar results about the convergence of minimizers in the~$Q$-tensor model.

\section{Setting of the general problem and preliminaries}
\label{sect: assumptions}

As we mentioned in the introduction, our asymptotic analysis will be carried out in a general setting, which recovers the Landau-de Gennes model \eqref{energy 1}--\eqref{energy 2} as a particular case. In this section, we detail the problem under consideration. The unknown is a function $\Omega \to \R^d$, where~$\Omega$ is a smooth, bounded (and possibly not simply connected) domain in $\R^2$. Let $g\colon \partial\Omega \to \R^d$ be a boundary datum, and define the Sobolev space $H^1_g(\Omega, \, \R^d)$ as the set of maps in $H^1(\Omega, \, \R^d)$ which agrees with $g$ on the boundary, in the sense of traces. We are interested in the problem
\begin{equation}
\min_{u\in H^1_g(\Omega; \, \R^d)}E_\varepsilon(u)
\label{energy}
\end{equation}
where 
\[
 E_\varepsilon(u) := E_\varepsilon(u, \, \Omega) = \int_\Omega \left\{\frac12\abs{\nabla u}^2 + \frac{1}{\varepsilon^2} f(u) \right\} 
\]
and $f\colon \R^d \to \R$ is a non negative, smooth function, satisfying the assumptions below. 

The existence of a minimizer for Problem \eqref{energy} can be easily inferred via the Direct Method in the calculus of variations, whereas we do not claim uniqueness. If $\ue$ denotes a
minimizer for $E_\varepsilon$, then $\ue$ is a weak solution of the Euler-Lagrange equation
\begin{equation}
-\Delta \ue + \frac{1}{\varepsilon^2}D f(\ue) = 0 \qquad \textrm{in } \Omega .
\label{EL}
\end{equation}
Via elliptic regularity theory, it can be proved that every solution of \eqref{EL} is smooth. 

\medskip
\textbf{Assumptions on the potential and on the boundary datum.} 
Denote, as usual, by $S^{d-1}$ the unit sphere of $\R^d$, and by $\dist(v, \, N)$ the distance between a point $v\in\R^d$ and a set $N$. We assume that $f\colon \R^d \to \R$ is a 
smooth function (at least of class $C^{2, \, 1}$), satisfying the following conditions:
\begin{enumerate}[label=\textup{ }{(H\arabic*)}\textup{ },ref={(H\arabic*)}]
 \item \label{hp: f} The function $f$ is non-negative, the set $\NN := f^{-1}(0)$ is non- empty, and $\NN$ is a smooth, compact and connected submanifold of $\R^d$, without boundary. We assume that $\NN$ is contained in the closed unit ball of $\R^d$.
 \item \label{hp: df} There exist some positive constants $\delta_0 < 1, \, m_0$ such that, for all $v\in \NN$ and all normal vector $\nu\in\R^d$ to $\NN$ at the point $v$,
\[
 Df(v + t\nu) \cdot \nu \geq m_0 t, \qquad \textrm{if } 0 \leq t \leq \delta_0 .
\]
 \item \label{hp: growth} For all $v\in\R^d$ with $\abs{v} > 1$, we have
\[
 f(v) > f\left(\frac{v}{\abs{v}}\right) .
\]
\end{enumerate}
The set $\NN$ will be referred as the vacuum manifold.
Concerning the boundary datum, we assume
\begin{enumerate}[label=\textup{ }{(H\arabic*)}\textup{ }, ref={(H\arabic*)}, resume] 
 \item \label{hp: g} $g\colon \partial\Omega \to \R^d$ is a smooth function, and $g(x)\in \NN$ for all $x\in\partial\Omega$.
\end{enumerate}

For technical reasons, we impose a restriction on the homotopic structure of $\NN$. A word of clarification: by conjugacy class in a group $G$, we mean any set of the form $\{axa^{-1}\colon a\in G\}$, for $x\in G$.
\begin{enumerate}[label=\textup{ }{(H\arabic*)}\textup{ }, ref={(H\arabic*)}, resume] 
 \item \label{hp: N} Every conjugacy class in the fundamental group of $\NN$ is finite. 
\end{enumerate}

\begin{remark}
\label{remark: hp}
The assumption \ref{hp: df} holds true if, at every point $v\in \NN$, the Hessian matrix $D^2 f(v)$ restricted to the normal space of $\NN$ at $v$ is positive definite. Hence, \ref{hp: df} may be interpreted as a \emph{non-degeneracy} condition for $f$, in the normal directions.
\end{remark}

We can provide a sufficient condition, in terms of the derivative of $f$, for \ref{hp: growth} as well: namely,
\[
 v\cdot D f(v) > 0 \qquad \textrm{for } \abs{v} > 1
\]
(indeed, this implies that the derivative of $t\in [1, \, +\infty) \mapsto f(tv)$ is positive). 
Hypothesis \ref{hp: growth} is exploited uniquely in the proof of the $L^\infty$ bound for the minimizer $\ue$.  

 Assumption \ref{hp: N} is trivially satisfied if the fundamental group $\pi_1(\NN)$ is abelian or finite. This covers many cases, arising from other models in condensed matter physic: Besides rod-shaped molecules in nematic phase, we mention planar spins ($\NN \simeq S^1$) and ordinary spins ($\NN \simeq S^2$), biaxial molecules in nematic phase ($\NN \simeq SU(2)/H$, where $H$ is the quaternion group), superfluid He-3, both in dipole-free and dipole-locked phases ($\NN \simeq (SU(2)\times SU(2))/H$ and $\NN \simeq \mathbb P^3(\R)$, respectively).

\medskip
\textbf{The Landau-de Gennes model.}
In this model, the configuration parameter belongs to the set $\Sz$ of matrices, given by
\[
\Sz := \left\{ Q\in M_3(\R)\colon Q^T = Q  , \: \tr Q = 0 \right\} .
\]
This is a real linear space, whose dimension, due to the symmetry and tracelessness constraints, is readily seen to be five. The tensor contraction $Q:P = \sum_{i, j} Q_{i j}P_{i j}$ defines a scalar product on $\Sz$, and the corresponding norm will be denoted $\abs{\,\cdot\,}$.
Clearly $\Sz$ can be identified, up to an equivalent norm, with the Euclidean space $\R^5$.

The bulk potential is given by
\begin{equation}
f(Q) := k - \frac a2 \tr Q^2 - \frac b3 \tr Q^3 + \frac c4 \left( \tr Q^2\right)^2 \qquad \textrm{for all } Q\in\Sz  ,
\label{LG f}
\end{equation}
where $a, \, b, \, c$ are positive parameters and $k$ is a properly chosen constant, such that $\inf f = 0$. (We have set $a := - \alpha(T - T_*)$ in formula \eqref{energy 2}). It is clear that the minimization problem \eqref{energy} does not depend on the value of $k$. This model is considered in detail, for instance, in \cite{MZ}, where $\Omega$ is assumed to be a bounded domain of~$\R^3$. 

In the Euler-Lagrange equation for this model, $Df$ has to be intended as the intrinsic gradient with respect to $\Sz$. Since the latter is a proper subspace of the $3 \times 3$ real matrices, $Df$ contains an extra term, which acts as a Lagrange multiplier associated with the tracelessness constraint. Therefore, denoting by $\Qe$ any minimizer, Equation \eqref{EL} reads
\begin{equation}
-\varepsilon^2\Delta \Qe  - a \Qe  - b\left\{\Qe ^2 - \frac13(\tr\Qe ^2)\Id\right\} + c\Qe \tr\Qe ^2 = 0 ,
\label{LG EL}
\end{equation}
where $\frac 13b \tr\Qe ^2 \Id$ is the Lagrange multiplier. We will show in Subsection~\ref{standard LG} that this problem fulfills \ref{hp: f}--\ref{hp: N}, and thus can be recovered in the general setting.

\subsection{Energy cost of a defect}
\label{subsect: cost}

By the theory of continuous media, it is well known (see \cite{Mermin}) that topological defects of codimension two are associated with homotopy classes of loops in the vacuum manifold $\NN$. Now, following an idea of \cite{Ch}, we are going to associate to each homotopy class a non negative number, representing the energy cost of the defect.

Let $\Gamma(\NN)$ be the set of free homotopy classes of loops $S^1 \to\NN$, that is, the set of the path-connected components of $C^0(S^1, \, \NN)$ --- here, ``free'' means 
that no condition on the base point is imposed. As is well-known, for a fixed base point $v_0\in \NN$ there exists a one-to-one and onto correspondence between $\Gamma(\NN)$ and the conjugacy classes of the fundamental group $\pi_1(\NN, \, v_0)$. As the latter might not be abelian, the set $\Gamma(\NN)$ is not a group, in general. Nevertheless, the composition of paths (denoted by $*$) induces a map
\begin{equation} \label{product}
\Gamma(\NN)\times \Gamma(\NN) \to \mathscr P\left(\Gamma(\NN)\right), \qquad (\alpha, \, \beta) \mapsto \alpha\cdot \beta
\end{equation}
in the following way: for each $v\in\NN$, fix a path $c_v$ connecting $v_0$ to $v$. Then, for $\alpha$, $\beta\in\Gamma(\NN)$ define
\[
 \alpha \cdot \beta := \left\{ \textrm{homotopy class of the loop } ((c_{f(1)} * f) * \widetilde{c_{f(1)}} ) * ((c_{g(1)} * g) * \widetilde{c_{g(1)}} ) \colon f\in\alpha, \, g\in\beta \right\}  ,
\]
where $\widetilde{c_{f(1)}}, \, \widetilde{c_{g(1)}}$ are the reverse paths of $c_{f(1)}, \, c_{g(1)}$ respectively. If we regard $\alpha$, $\beta$ as conjugacy classes in~$\pi_1(\NN, \, v_0)$, we might check that
\[
\alpha\cdot\beta = \left\{\textrm{conjugacy class of } ab \colon a\in\alpha, \, b\in\beta \right\}
\]
(in particular, we see that $\alpha\cdot\beta$ does not depend on the choice of $(c_v)_{v\in\NN}$). As $\alpha$, $\beta$ are finite, due to \ref{hp: N}, the set $\alpha * \beta$ is finite as well.

The set $\Gamma(\NN)$, equipped with this product, enjoys some algebraic properties, which descend from the group structure of $\pi_1(\NN, \, v_0)$. The resulting structure is referred to as the polygroup of conjugacy classes of $\pi_1(\NN, \, v_0)$, and was first recognized by Campaigne (see \cite{Camp}) and Dietzman (see \cite{Diat}). We remark that, even if $\pi_1(\NN, \, b)$ is not abelian, we have $\alpha \cdot \beta = \beta \cdot \alpha$ for all $\alpha, \, \beta\in \Gamma(\NN)$. This follows from $a b = a(b a) a^{-1}$, which holds true for all $a, \, b\in \pi_1(\NN, \, v_0)$.

The geometric meaning of the map \eqref{product} is captured by the following proposition. By convention, let us set~$\prod_{i = 1}^1 \gamma_i := \{\gamma_1\}$.

\begin{lemma} \label{lemma: extension}
 Let $D$ be a smooth, bounded domain in $\R^2$, whose boundary has $k\geq 2$ connected components, labeled $C_1, \, \ldots, \, C_k$. For all 
 $i = 1, \ldots, k$, let $g_i\colon C_i \to \NN$ be a smooth boundary datum, whose free homotopy class is denoted by $\gamma_i$. If the condition
\begin{equation}
\prod_{i = 1}^h \gamma_i \cap \prod_{i = h+1}^k \gamma_i \neq \emptyset \, 
\label{extension}
\end{equation}
holds for some index $h$, then there exists a smooth function $g\colon \overline D \to \NN$, which agrees with $g_i$ on every $C_i$. Conversely, if  such an extension exists then 
the condition \eqref{extension} holds for all $h \in \{ 1, \, \ldots, \, k\}$.
\end{lemma}
\proof
Throughout the proof, given a path $c$ we will denote the reverse path by $\widetilde c$.

Assume that \eqref{extension} holds. We claim that the boundary data can be extended continuously on $D$. It is convenient to work out the construction in the subdomain
\[
D' := \left\{ x\in D \colon \dist(x, \, D) > \delta\right\}  ,
\]
where $\delta > 0$ is small, so that $D$ and $D'$ have the same homotopy type.
Up to a diffeomorphism, we can suppose that $D'$ is a disk with $k$ holes, and $C_1$ is the exterior boundary. It is equally fair to assume that there exists a path $B$, homeomorphic to a circle, which splits $D'$ into two regions, $D_1$ and $D_2$, with
\[
 \partial D_1 = B \cup \bigcup_{i = 1}^h C_i , \qquad \partial D_2 = B \cup \bigcup_{i = h + 1}^k C_i  .
\]
This configuration is illustrated in the Figure~\ref{fig: patate}.
Let $b\colon B \to \NN$ be a loop whose free homotopy class belongs to $\prod_{i = 1}^h \gamma_i \cap \prod_{i = h+1}^k \gamma_i$. 

\begin{figure}[t] \label{fig: patate}
\begin{center}
\includegraphics[width =.65\textwidth, keepaspectratio = true]{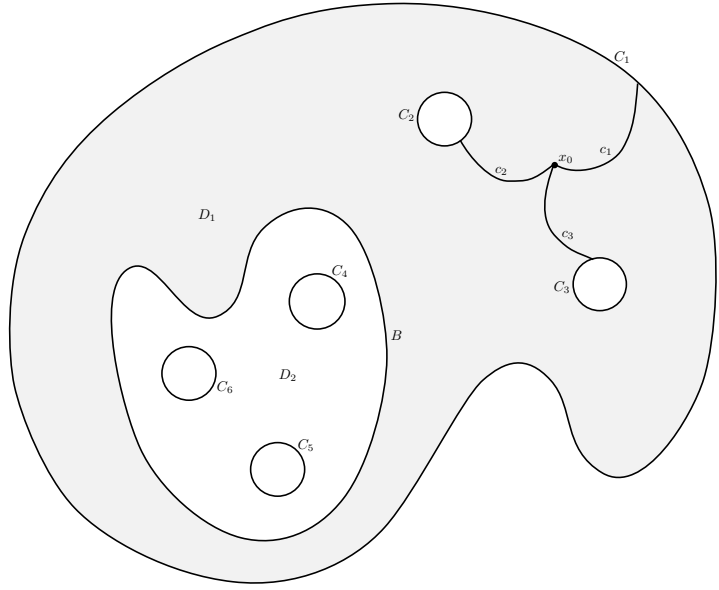}
\end{center}
\caption{The geometry of $D'$ in Lemma~\ref{lemma: extension}.}
\end{figure}

We wish, at first, to extend the boundary data to a continuous function defined on $D_1$. Let $c_1, \, \ldots, \, c_h$ be mutually non intersecting paths $[0, \, 1] \to D_1$, connecting a fixed base point $x_0\in D_1$ with $C_1, \ldots, \, C_h$ respectively, and let $\Sigma$ denote the union of $C_1, \ \ldots, \, C_h$ and the images of $c_1, \, \ldots, \, c_h$. The set $\Sigma$ can be parametrized by the loop
\[
 \alpha := \left((c_1*\alpha_1)*\widetilde{c_1}\right) * \left((c_2*\alpha_2)*\widetilde{c_2}\right) * \cdots * \left((c_h*\alpha_h)*\widetilde{c_h}\right)  ,
\]
where $\alpha_i\colon [0, \, 1]\to C_i$ is a parametrization of $C_i$ proportional to arc length.

Next, we ``push forward'' $\alpha$ to a loop in $\NN$. Since $b\in\prod_{i = 1}^h\gamma_i$, there exists a loop $\sigma$, freely homotopic to $b$, which can be written as
\[
  \sigma := \left((\sigma_1*g_1')*\widetilde{\sigma}_1\right) * \left((\sigma_2*g_2')*\widetilde{\sigma}_2\right) * \cdots * \left((\sigma_h*g_h')*\widetilde{\sigma}_h\right)  ,
\]
where $g_i'\in\gamma_i$ and $\sigma_i$ is a path in $\NN$ connecting a fixed base point $v_0\in\NN$ with $g_i'(1)$, for each $i\in\{1, \, 2, \, \ldots, \, h \}$.
We can regard $\sigma$ as a map $\Sigma \to \NN$: more precisely, we can set $t \in [0, \, 1] \mapsto \sigma(\alpha^{-1}(t))$ and check that this mapping is well-defined. By construction, there exists a homotopy between $b$ and $\sigma$, which provides a continuous extension of the boundary data $g_1', \ldots, \, g_h', \, b$ to a mapping $v_1\colon D_1\to \NN$.

We perform the same construction on the subdomain $D_2$, obtaining a continuous function $v_2$. Pasting $v_1$ and~$v_2$ we get a continuous map $v'\colon \overline D' \to \NN$, whose trace on each $C_i$ is homotopic to $g_i$. As $D\setminus D'$ is just a small neighborhood of $\partial D$, it is not difficult to extend $v'$ to a continuous function $v\colon \overline D \to \NN$, such that $\left. v\right|_{C_i} = g_i$ for all $i$. Smoothness can be recovered, for instance, via a standard approximation argument.

Conversely, assume that an extension $g$ exists, and let $B, \, D_1, \, D_2, \, \Sigma$ be as before, for $h$ arbitrary. Then, $\left.g\right|_{D_1}$ provides a free homotopy between $\left. g\right|_B$ and $\left. g\right|_\Sigma$, so the homotopy class of $\left. g\right|_B$ belongs to $\prod_{i = 1}^h\gamma_i$. Similarly, the class of $\left. g\right|_B$ belongs to $\prod_{i = h + 1}^k\gamma_i$, and hence the condition~\ref{extension} holds.
\endproof

For each $\gamma\in \Gamma(\NN)$, we define its length as
\begin{equation}
\lambda(\gamma) := \inf\left\{ \left(2\pi\int_{S^1}\abs{c'(\theta)}^2 \, \d\theta\right)^{1/2}\colon c\in \gamma\cap H^1(S^1, \, \NN)\right\}  .
\label{lambda}
\end{equation}
First, the set $\gamma\cap H^1(S^1, \, \NN)$ is not empty since the embedding $H^1(S^1, \, \NN)\hookrightarrow C^0(S^1, \, \NN)$ is compact and dense. Then, notice the infimum in 
\eqref{lambda} is achieved, and all the minimizers $c$ are geodesics. Thus, $\abs{c'}$ is constant, and $\lambda(\gamma) = 2\pi\abs{c'}$ coincides with the length of a minimizing geodesic.

In the definition of the energy cost of a defect, it is convenient take into account the product we have endowed $\Gamma(\NN)$ with. For each $\gamma\in \Gamma(\NN)$ we set
\begin{equation}
\lambda_*(\gamma) := \inf\left\{ \frac{1}{4\pi}\sum_{i = 1}^k \lambda(\gamma_i)^2 \colon k\in\N, \, \: \gamma_i\in \Gamma(\NN), \, \gamma\in\prod_{i = 1}^k\gamma_i \right\}  ,
\label{lambda *}
\end{equation}
where the order of the product is not relevant.
It is worth pointing out that the infimum in~\eqref{lambda *} is, in fact, a minimum. Indeed, since $\NN$ is compact manifold, its fundamental group is finitely generated; on the other hand, $\gamma$ contains only a finite number of elements of $\pi_1(N, \, v_0)$, by \ref{hp: N}. As a result, we see that the infimum in \eqref{lambda *} is computed over finitely many $k$-uples $(\gamma_1, \, \ldots, \gamma_k)$.

Roughly speaking, the number $\lambda_*(\gamma)$ can be regarded as the energy cost of the defect $\gamma$.
For example, when $\NN = S^1$ we have $\Gamma(S^1) \simeq \pi_1(S^1)\simeq \mathbb Z$, that is, the homotopy classes in $\Gamma(S^1)$ are completely determined by their degree $d\in\mathbb Z$.
Besides, $\lambda(d) = 2\pi\abs{d}$ and $\lambda_*(d) = \pi \abs{d}$, the infimum in \eqref{lambda *} being reached by the decomposition
\[
\gamma_1 = \gamma_2 = \cdots = \gamma_{\abs{d}} = \textrm{sign }d = \pm 1  .
\]
Hence, in this case decomposing the defect is energetically favorable. This is related to the quantization of singularities in the Ginzburg-Landau model (see \cite{BBH}). 

By definition, $\lambda_*$ enjoys the useful property
\begin{equation}
 \lambda_*(\gamma) \leq \sum_{i = 1}^k \lambda_*(\gamma_i) \qquad \textrm{if} \quad \gamma\in\prod_{i = 1}^k \gamma_i \quad \textrm{with} \quad \gamma_i\in\Gamma(\NN).
 \label{lambda * <}
\end{equation}

We conclude this subsection by coming back to our main problem \eqref{energy}, and fixing some notation that will be used throughout this work. 
\begin{defn} \label{def: trivial}
A continuous function $g\colon \partial \Omega \to \NN$ will be called \emph{homotopically trivial} if and only if it can be extended to a continuous function $\overline\Omega\to \NN$.
\end{defn}
In case $\Omega$ is a simply connected domain, thus homeomorphic to a disk, being homotopically trivial is equivalent to being null-homotopic, that is, being homotopic to a constant. By contrast, these notions do not coincide any longer for a general domain. For instance, suppose that $\Omega$ is an annulus, bounded by two circles $C_1$ and $C_2$, and that $g_1, \, g_2$ are smooth data, defined on $C_1, \, C_2$ respectively and taking values in $\NN$. If $g_1, \, g_2$ are in the same homotopy class, then the boundary datum is homotopically trivial in the sense of the previous definition, although each $g_i$, considered in itself, might not be null-homotopic. We will provide a characterization of homotopically trivial boundary data, for general domains, with the help of the tools we have described in this section.

Label the connected components of $\partial\Omega$ as $C_1, \, \ldots, \, C_k$, and denote by $\gamma_i$, for $i \in \{ 1, \, \ldots, k\}$, the free homotopy class of the boundary datum $g$ restricted to $C_i$. Define
\[
\kappa_* := \inf\left\{ \lambda_*(\gamma)\colon \gamma\in\prod_{i = 0}^k \gamma_i\right\} ,
\]
where $\lambda_*$ has been introduced in \eqref{lambda *}. By definition of $\lambda_*$, we have
\begin{equation}
\kappa_* = \inf\left\{\frac{1}{4\pi}\sum_{j = 1}^m \lambda(\eta_j)^2\colon m\in\N  , \: \eta_j \in \Gamma(\NN)  , \: \prod_{j = 1}^m \eta_j \cap \prod_{i = 1}^k\gamma_i \neq \emptyset\right\}  .
\label{kappa *}
\end{equation}
In both formulae, the infima are taken over finite sets, and hence are minima. 

As a straightforward consequence of Lemma~\ref{lemma: extension}, we obtain the following result, characterizing trivial boundary data. The proof is left to the reader.

\begin{cor}
 Let $D\subseteq\R^2$ be a smooth, bounded domain, 
 and let $g_i$, $\gamma_i$ be as in Lemma~\ref{lemma: extension}. Then, the following conditions are equivalent:
 \begin{itemize}
  \item[(i)] the boundary datum $(g_i)_{i = 1}^k$ is homotopically trivial;
  \item[(ii)] denoting by $\mathbf{\epsilon}$ the free homotopy class of any constant map in $\NN$, we have
  \[
   \mathbf{\epsilon} \in \prod_{i = 1}^k \gamma_i \, ;
  \]
  \item[(iii)] $\kappa_* = 0$.
 \end{itemize}
\end{cor}

\subsection{The nearest point projection onto a manifold}
\label{subsect: geometric lemma}

In this subsection, we discuss briefly a geometric tool which will be exploited in our analysis: the nearest point projection on a manifold.

Let $\NN$ be a compact, smooth submanifold of $\R^d$, of dimension $n$ and codimension $k$ (that is, $d = n + k$). It is well known (see, for instance, \cite[Chapter 3, p. 57]{Moser})
that there exists a neighborhood $U$ of $\NN$ with the following property: for all $v\in U$, there exists a unique point $\pi(v)\in \NN$ such that
\begin{equation}
\abs{v - \pi(v)} = \dist(v, \, \NN) .
\label{def npp}
\end{equation}
The mapping $v\in U \mapsto \pi(v)$, called the nearest point projection onto $\NN$, is smooth, provided that $U$ is small enough. Moreover, $v - \pi(v)$ is a normal vector to $\NN$ at each point $v\in\NN$ (all this facts are proved, e.g., in \cite{Moser}).

Throughout this work, we will assume that $\pi$ is well-defined and smooth on the $\delta_0$-neigh\-bor\-hood of $\NN$, where $\delta_0$ is introduced in \ref{hp: df}.

\begin{remark} \label{remark: f}
With the help of $\pi$, we can easily derive from \ref{hp: df} some useful properties of $f$ and its derivatives. Let $v\in\R^d$ be such that $\dist(v, \, \NN) \leq \delta_0$. Then,
\[
 m_0 \dist(u, \, \NN) \leq Df(u) \cdot (u - \pi(u)) \leq M_0 \dist(u, \, \NN) .
\]
Indeed, the lower bound is given by \ref{hp: df}, whereas the upper bound is obtained by a Taylor expansion of $Df$ around the point $\pi(u)$ (remind that $Df(\pi(u)) = 0$ because $f$ is minimized on $\NN$). As $\NN$ is compact, the constant $M_0$ can be chosen independently of $v$. Via the fundamental theorem of calculus, we infer also
\[
 \frac12 m_0 \dist^2(u, \, \NN) \leq f(u) = \int_0^1 Df \left(\pi(u) + t(u - \pi(u)) \right) \cdot (u - \pi(u)) \, \d t \leq \frac12 M_0 \dist^2(u, \, \NN)  .
\]
\end{remark}

The following lemma establishes a gradient estimate for the projection of mappings.

\begin{lemma} \label{lemma: npp}
Let $u\in C^1(\Omega, \, \R^d)$ be such that $\dist(u(x), \, \NN) \leq \delta_0$ for all $x\in\Omega$, and define
\[
 \sigma(x) := \dist(u(x), \, \NN)  , \qquad v(x) := \pi(u(x)) 
\]
for all $x\in\Omega$. Then, the estimates
\begin{equation}
(1 - M \sigma) \abs{\nabla v}^2 \leq \abs{\nabla u}^2 \leq (1 + M \sigma) \abs{\nabla v}^2 + \abs{\nabla \sigma}^2
\label{npp}
\end{equation}
hold, for a constant $M$ depending only on $\NN$, $k$.
\end{lemma}

\proof
Fix a point $x\in\Omega$. Let $\nu_1, \, \nu_2, \, \ldots, \nu_k$ be a moving orthonormal frame for the normal space to $\NN$, defined on a neighborhood of $v(x)$. (Even if $\NN$ is not 
orientable, such a frame is locally well-defined). Then, for all $y$ in a neighborhood of $x$, there exist some numbers $\alpha_1(y), \, \alpha_2(y), \ldots, \, \alpha_k(y)$ such that
\begin{equation}
u(y) = v(y) + \sum_{i = 1}^k \alpha_i(y) \nu_i(v(y))  .
\label{NPP 1}
\end{equation}
The functions $v$, $\alpha_i$ are as regular as $u$. Differentiating the equation \eqref{NPP 1}, and raising to the square each side of the equality, we obtain
\begin{equation} \label{npp 2}
\begin{aligned}
\abs{\nabla u}^2 - \abs{\nabla v}^2  &= \sum_{i = 1}^k \left\{\alpha_i^2\abs{\nabla\nu_i(v)}^2 + \abs{\nabla\alpha_i}^2 \right. \\
\Big. &+2\alpha_i\nabla v : \nabla\nu_i(v) + 2\nabla v : \left(\nu_i(v)\otimes \nabla\alpha_i \right) + 2\alpha_i \nabla\nu_i(v) : \left(\nu_i(v)\otimes \nabla\alpha_i \right) \Big\}  .
\end{aligned}
\end{equation}
The fourth term in the right-hand side vanishes, because $\nabla v$ is tangent to $\NN$. The last term vanishes as well since, differentiating $\nu_i = 1$, we have 
$(\nabla\nu_i)\nu_i = 0$. For the first term of the right-hand side, we set
\[
M := 1 + \sup_{1 \leq i \leq k} \norm{\nabla \nu_i}_{L^\infty}^2 
\]
and we remark that
\[
 \sum_{i = 1}^k \alpha_i^2\abs{\nabla\nu_i(v)}^2 \leq M \sum_{i = 1}^k \alpha_i^2\abs{\nabla v}^2 = M \sigma^2 \abs{\nabla v}^2  .
\]
By the Cauchy-Schwarz inequality and $\sum_{i = 1}^k \alpha_i \leq C_k \left(\sum_{i = 1}^k\alpha_i^2\right)^{1/2}$, we can write
\begin{equation}
\abs{ \sum_{i = 1}^k \alpha_i\nabla v : \nabla\nu_i(v)} \leq M \sum_{i = 1}^k\alpha_i \abs{\nabla v}^2 \leq M \sigma \abs{\nabla v}^2  ,
\label{npp 4}
\end{equation}
up to modifying the value of $M$ in order to absorb the factor $C_k$. Furthermore, since $\sigma\leq \delta_0 < 1$, from \eqref{npp 2} and \eqref{npp 4} we infer
\begin{equation}
(1 - M \sigma) \abs{\nabla v}^2 + \sum_{i = 1}^k\abs{\nabla \alpha_i}^2 \leq \abs{\nabla u}^2 \leq (1 + M \sigma) \abs{\nabla v}^2 + \sum_{i = 1}^k\abs{\nabla \alpha_i}^2  .
\label{npp 3}
\end{equation}

The lower bound in \eqref{npp} follows immediately, and we only need to estimate the derivatives of $\alpha_i$ to conclude. It follows from \eqref{NPP 1} that 
$\alpha_i = (u - v)\cdot\nu_i(v)$.
Differentiating and raising to the square this identity, and taking into account that $(\nabla\nu_i)\nu_i = 0$, we deduce
\[
\sum_{i = 1}^k \abs{\nabla\alpha_i}^2 = \sum_{i = 1}^k \left\{\abs{\nabla(u - v) \cdot \nu_i(v)}^2 + \abs{(u - v) \cdot \nabla\nu_i(v)}^2 \right\}  .
\]
Then
\begin{equation} \label{npp 6}
\sum_{i = 1}^k \abs{\nabla\alpha_i}^2 \leq M \left\{ \abs{\nabla(u - v)}^2 + \sigma^2 \abs{\nabla v}^2 \right\} .
\end{equation}
Computing the gradient of $\sigma = \abs{u - v}$ by the chain rule yields $\abs{\nabla \sigma} =  \abs{\nabla (u - v)}$. Therefore, the estimates \eqref{npp 4} and \eqref{npp 6} imply 
the upper bound in \eqref{npp}.

Notice that our choice of the constant $M$ depends on the neighborhood where the frame $(\nu_i)_{1\leq i \leq k}$ is defined. 
However, since $\NN$ is compact, we can find a constant for which the inequality \eqref{npp} holds globally.
\endproof

 \section{Biaxiality phenomena in the Landau-de Gennes model}
 \label{sect: biaxial}

 We focus here on the Landau-de Gennes model \eqref{energy}--\eqref{LG f}. To stress that this discussion pertains to a specific case, throughout the section we use $Q$ instead of $u$ to denote the unknown. In constrast, the other notations --- the symbols for the potential and the vacuum manifold, in particular --- are still valid.
 
 In this section, we aim to prove Theorem~\ref{thm: intro-biaxial}. This can be achieved independently of the asymptotic analysis: we need only to recall a well-known property of minimizers.
 
 \begin{property}
  Any minimizer $\Qe$ for problem \eqref{energy}--\eqref{LG f} is of class $C^\infty$ and fulfills
  \[
  \norm{\Qe}_{L^\infty(\Omega)} \leq \sqrt{\frac23} s_*  ,
  \]
  where $s_*$ is the constant defined in \eqref{s+}.
 \end{property}

 The proof will be given further on (see Lemma~\ref{lemma: L infty}).
 
 \subsection{Useful properties of \emph{Q}-tensors}
\label{standard LG}

Our first goal is to show that the Landau-de Gennes model satisfies \ref{hp: f}--\ref{hp: N}, so that it fits to our general setting. In doing so, we recall some classical, useful facts about $Q$-tensors.
Let us start by the following well-known definition: we set
\[
\beta(Q) := 1 - 6\frac{\left(\tr Q^3\right)^2}{\left(\tr Q^2\right)^3} \qquad \textrm{for } Q \in \Sz\setminus \{0\}  .
\]
This defines a smooth, homogeneous function $\beta$, which will be termed the biaxiality parameter. It could be proved that $0\leq \beta \leq 1$ (see, for instance, \cite[Lemma 1 and Appendix]{MZ} and the references therein). Now, we can precise what we mean by ``maximally biaxial minimizers'', an expression we have defined informally in the Introduction.

\begin{defn} \label{def: biaxial} Let $Q$ be a function in $H^1_g(\Omega, \, \Sz)$. We say that $Q$ is \emph{almost uniaxial} if and only if
\[
  \max_{\overline\Omega} \beta(Q) < 1  .
\]
Otherwise, we say that $Q$ is maximally biaxial.
\end{defn}

Another classical fact about $Q$-tensors is the following representation formula, which turns out to be useful in several occasions. 

\begin{lemma} \label{lemma: representation}
 For all fixed $Q\in\Sz\setminus \{0\}$, there exist two numbers $s \in (0, \, +\infty)$, $r\in [0, \, 1]$ and an orthonormal pair of vectors $(n, \, m)$ in $\R^3$ such that
 \begin{equation} \label{representation}
  Q = s \left\{ n^{\otimes 2} - \frac 13 \Id + r \left( m^{\otimes 2} - \frac13 \Id \right) \right\} .
 \end{equation}
 Furthermore, labeling the eigenvalues of $Q$ as $\lambda_1 \geq \lambda_2 \geq \lambda_3$, if $(s, \, r, \, n, \, m)$ satisfies the conditions above then
 \begin{equation} \label{representation 2}
 s = 2\lambda_1 + \lambda_2  , \qquad r = \frac{\lambda_1 + 2\lambda_2}{2\lambda_1 + \lambda_2} , 
 \end{equation}
and $n, \, m$ are eigenvectors associated to $\lambda_1, \, \lambda_2$ respectively. 
\end{lemma}
\begin{proof}[Sketch of the proof]
Let $(s, \, r, \, n, \, m)$ be a set of parameters with the desired properties, and denote by $p$ the vector product of $n$ and $m$, so that $(n, \, m, \, p)$ is a positive orthonormal basis of $\R^3$. Exploiting the identity $\Id = n^{\otimes 2} + m^{\otimes 2} + p^{\otimes 2}$, we can rewrite \eqref{representation} as
\[
 Q = \frac s3 (2 - r) n^{\otimes 2} + \frac s3 (r - 1) m^{\otimes 2} - \frac s3 (1 + r) p^{\otimes 2} .
\]
The constraints $s \geq0$, $0 \leq r \leq 1$ entail
\[
 \frac s3 (2 - r) \geq \frac s3 (r - 1) \geq - \frac s3 (1 + r) .
\]
We conclude that
\[
 \lambda_1 = \frac s3 (2 - r) , \qquad \lambda_2 = \frac s3 (r - 1) , \qquad \lambda_3 = -\frac s3 (1 + r) ,
\]
and that $n, \, m, \, p$ are eigenvectors associated to $\lambda_1, \, \lambda_2, \, \lambda_3$ respectively. The identities \eqref{representation 2} follow by straightforward computations. Conversely, it is easily checked that the parameters defined by \eqref{representation 2} satisfy \eqref{representation}.
\end{proof}

\begin{remark} \label{remark: beta lambda}
The limiting cases $r = 0$ and $r = 1$ correspond, respectively, to $\lambda_1 = \lambda_2$ and $\lambda_2 = \lambda_3$. In the literature, these cases are sometimes referred to as prolate and oblate uniaxiality, respectively. The modulus and the biaxiality parameter of $Q$ can be expressed in terms of $s, \, r$ as follows (compare, for instance, \cite[Equation (187)]{MZ}):
 \[
  \abs{Q}^2 = \frac23 s^2 \left(r^2 - r + 1 \right) , \qquad \beta(Q) = \frac{27r^2\left(1 - r\right)^2}{4\left(r^2 - r + 1 \right)^3} .
 \]
 In particular, one has
 \begin{equation} \label{s abs}
  s(Q) \geq \sqrt{\frac 32} \abs{Q} .
 \end{equation}
 Also, remark that $Q$ is maximally biaxial if and only if there exists a point $x\in\Omega$ such that $r(Q(x)) = 1/2$.
\end{remark}

With the help of \eqref{representation}, the set of minimizers of the potential $f$ can be described as follows.

\begin{prop} \label{prop: lower estimate f}
Let $f$ be given by \eqref{LG f}, and set
\begin{equation} \label{s+}
 s_* :=\frac{1}{4c}\left\{b + \sqrt{b^2 + 24ac}\right\} .
\end{equation}
Then, the minimizers for $f$ are exactly the matrices which can be expressed as
\[
Q = s_* \left(n^{\otimes 2} - \frac 13 \Id\right) \qquad \textrm{for some unit vector } n\in \R^3 .
\]
The set of minimizers is a smooth submanifold of $\Sz$, homeomorphic to the projective plane~$\PR$, contained in the sphere $\left\{Q\in\Sz\colon \abs{Q} = s_*\sqrt{2/3} \right\}$. In addition, $\beta(Q) = 0$ if $Q$ is a minimizer for $f$.
\end{prop}

The reader is referred to \cite[Propositions 9 and 15]{MZ} for the proof. Here, we mention only that a diffeomorphism $\PR \to \NN$ can be constructed by considering the map
\begin{equation} \label{phi}
 \phi\colon n\in S^2 \mapsto s_* \left( n^{\otimes 2} - \frac13 \Id \right)
\end{equation}
and quotienting it out by the universal covering $S^2 \to \PR$, which is possible because $\phi(n) = \phi(-n)$. Remark also that, for all $n\in S^2$ and all tangent vector $v\in T_n S^2$, we have
\begin{equation} \label{d phi}
 \left\langle d\phi(n), \, v\right\rangle = s_* \left( n \otimes v + v \otimes n \right)  ,
\end{equation}
as it is readily seen differentiating the function $t \mapsto \phi(n + tv)$.

It is well-known that the fundamental group of the real projective plane consists of two elements only. Therefore, $\Gamma(\PR) \simeq \pi_1(\PR) \simeq \mathbb Z/2\mathbb Z$, and \ref{hp: N} is trivially satisfied.
Hypothesis \ref{hp: f} is fulfilled as well, up to rescaling the norm in the parameter space, so that the vacuum manifold is contained in the unit sphere. Let us perform such a scaling: we associate to each map $Q\in H^1(\Omega, \, \Sz)$ a rescaled function $Q_*$, by
 \[
   Q(x) = \sqrt{\frac23} s_* Q_*(x) \qquad \textrm{for all } x\in\Omega .
 \]
 It is easily computed that
 \[
  E_\varepsilon(Q) = \frac{2s_*^2}{3} \int_\Omega \left\{ \frac12 \abs{\nabla Q_*}^2 + \frac{1}{\varepsilon^2} f^*(Q_*)\right\} ,
 \]
 where $f^*$ is given by
 \begin{equation} \label{rescaled f}
  f^*(Q_*) := \frac{3}{2s_*^2} f\left(\sqrt{\frac23}s_* Q_*\right) = -\frac{a^*}{2} \tr Q_*^2 - \frac{b^*}{3} \tr Q_*^3 + \frac{c^*}{4} \left(\tr Q_*^2 \right)^2 
 \end{equation}
 and
 \[
  a^* := a , \qquad b^* := \sqrt{\frac23}s_* b , \qquad c^* := \frac23 s_*^2 c .
 \]
 Notice that $f^*$ is minimized by $s = \sqrt{3/2}, \, r = 0$. Thus, we can assume that the vacuum manifold is contained in the unit sphere of $\Sz$, up to substituting $f^*$ for $f$ in the 
 minimization problem \eqref{energy}.

\begin{lemma} \label{lemma: conti}
 The potential $f^*$ defined by \eqref{rescaled f} fulfills \ref{hp: f}--\ref{hp: growth}.
\end{lemma}
\proof
We know by Proposition~\ref{prop: lower estimate f} that \ref{hp: f} is satisfied. We compute the gradient of $f^*$:
\[
 Df^*(Q) = -a^* Q - b^* Q^2 - \frac13 b^* (\tr Q^2) \Id + c^* Q \tr Q^2 ,
\]
where the term $-b^* (\tr Q^2)\Id/3$ is a Lagrange multiplier, accounting for the tracelessness constant in the definition of $\Sz$. With the help of Remark~\ref{remark: hp}, we show that \ref{hp: growth} is fulfilled as well. Indeed, one cas use the inequality $\sqrt6 \tr Q^3 \leq \abs{Q}^3$ to derive
\[
 Df^*(Q) : Q  =  -a^* \abs{Q}^2 - b^* \tr Q^3 + c^* \abs{Q}^4 \geq - a \abs{Q}^2 - \frac{s_*b}{3} \abs{Q}^3 + \frac{2s_*^2c}{3} \abs{Q}^4 .
\]
It is readily seen that the right-hand side is positive, for $\abs{Q}^2 > 1$.

Finally, let us check the condition \ref{hp: df}. For a fixed $Q\in \NN$, there exists $n\in S^2$ such that $Q = \phi(n)$, where $\phi$ is the smooth mapping defined by \eqref{phi}. Up to rotating the coordinate frame, we can assume without loss of generality that $n = e_3$. By formula \eqref{d phi}, we see that the vectors
\[
 X_k = \sqrt{\frac32} (e_k\otimes e_3 + e_3 \otimes e_k) , \qquad k \in \{1, \, 2\}
\]
form a basis for the the tangent plane to $\NN$ at the point $Q$. As a consequence, $P\in\Sz$ is a normal vector to $\NN$ at $Q$ if and only if $P : X_1 = P : X_2 = 0$ or, equivalently, iff it has the form
\begin{equation} \label{forma p}
P = \left[ \begin{matrix}
            p_1 & p_2 & 0 \\
            p_2 & p_3 & 0 \\
            0   & 0   & - p_1 - p_3
           \end{matrix}
 \right]  .
\end{equation}
It is easily checked that $PQ = QP$. Now, we compute
\begin{align*}
 Df^*(Q + tP):P &= - a^*(Q + tP):P - b^*(Q + tP)^2:P + c^* \tr(Q + tP)^2 (Q + tP):P \\
                &= t \left\{ -a^* \abs{P}^2 - 2 b^* (PQ):P + 2c^* (\tr PQ)^2  + c^* \abs{P}^2 \right\} + O(t^2)
\end{align*}
(we have used that $Df^*(Q) = 0$ and that $Q:P = \tr PQ$). By \eqref{forma p}, we have
\[
(PQ) : P = \frac13 \sqrt{\frac32} \left\{2(p_1 + p_3)^2 - p_1^2 - p_3^2 - 2p_2^2 \right\} \leq \frac12 \sqrt{\frac32}(p_1 + p_3)^2 , \qquad \tr PQ = - \sqrt{\frac32} (p_1 + p_3)
\]
and hence
\[
Df^*(Q + tP) : P \geq t\left(-a^* + c^* \right)\abs{P}^2 + t \left( 3c^* - \sqrt{\frac32}b^* \right)(p_1 + p_3)^2 + O(t^2) .
\]
The coefficients in the right-hand side are readily shown to be non negative: more precisely, 
\[
-a^* + c^* = \frac{1}{12c} \left\{ b^2 + b \sqrt{b^2 + 24ac}\right\} > 0 , \qquad
3c^* - \sqrt{\frac32}b^* = \frac12 s_*\left\{\sqrt{b^2 + 24ac} - b \right\} \geq 0 .
\]
Thus, we have proved that $\partial_P^2 f(Q) \geq - a^* + c^* > 0$. 
\endproof

Finally, let us point out a metric property of $\NN$. We know that the parameter $\kappa_*$, defined by \eqref{kappa *}, can take a unique positive value, corresponding to a geodesic loop in $\NN$ which generates $\pi_1(\NN)$. Since we aim to calculate it, we need to characterize the geodesic of $\NN$. Some help is provided by the mapping $\phi$ that we have introduced in \eqref{phi}.

\begin{lemma}
 For all $n\in S^2$ and all tangent vector $v\in T_n S^2$, it holds that
 \begin{equation} \label{metric}
  \abs{\left\langle d\phi(n), \, v\right\rangle} = \sqrt2 s_* \abs{v} ,
 \end{equation}
 that is, the differential of $\phi$ is a homothety at every point. In addition, the geodesics of $\NN$ are exactly the images via $\phi$ of the geodesics of $S^2$, that is, great circles parametrized proportionally to arc length.
\end{lemma}
\proof
It follows plainly from \eqref{d phi} that
\[
 \abs{\left\langle d\phi(n), \, v\right\rangle}^2 = 
 2s_*^2 \sum_{i, \, j}\left(n_i v_j n_i v_j + n_i v_j v_i n_j\right) .
\]
As $v$ is tangent to the sphere at the point $n$, we have $v \cdot n = 0$, so the second term in the summation vanishes, and we recover \eqref{metric}.

Denote by $g, \, h$ the first fundamental forms on $S^2, \, \NN$ respectively (that is, the metrics these manifolds inherit from being embedded in an Euclidean space). In terms of pull-back metrics, Equation \eqref{metric} reads
\[
 \phi^* h = 2 s_*^2 g
\]
and, since the scaling factor $2s_*^2$ is constant, the Levi-Civita connections associated with $\phi^*h$ and $g$ coincide: this can be argued, for instance, from the standard expression for Christoffel symbols
\[
    \Gamma_{jk}^i=\frac12 g^{il} \left( \frac{\partial g_{lj} }{\partial x^k} +\frac{\partial g_{lk}}{\partial x^j}  -\frac{\partial g_{jk}}{\partial x^l}  \right)  .
\]
As a consequence, we derive the characterization of geodesics in $\NN$.
\endproof

Proving the following result is not difficult, once we know what the geodesics in $\NN$ are. The proof is left to the reader.

\begin{cor} \label{cor: bordo geodesico}
The curve 
 \begin{equation} \label{bordo geodesico}
 c_* \colon \theta\in[0, \, 2\pi] \mapsto \sqrt{\frac32} \left( n_0(\theta)^{\otimes 2} - \frac13 \Id\right)  ,
 \end{equation}
 where $n_0(\theta) = \left[\cos(\theta/2), \, \sin(\theta/2), \, 0\right]^T$, minimizes the functional
 \[
  c \in H^1(S^1, \, \NN) \mapsto \frac{1}{2} \int_0^{2\pi} \abs{c'(\theta)}^2 \, \d\theta
 \]
 among the non-homotopically trivial loops in $\NN$. In particular,
 \[
  \kappa_* = 
  \frac34 \pi  .
 \]
\end{cor}

 \subsection{Energy estimates for an almost uniaxial map}
 
 In this subsection, we aim to establish a lower estimate for the energy of uniaxial maps. As a first step, we examine the potential $f^*$, and we provide the following result, which improves slightly \cite[Proposition 9]{MZ}.

 \begin{lemma} \label{lemma: f below LG}
 For all $Q\in \Sz$ with $\abs{Q} \leq 1$, the bulk potential is bounded by
 \begin{equation} \label{f below LG 1}
  f^*(Q) \geq \mu_1 \left(1 - \abs{Q}\right)^2 + \sigma \beta(Q)\abs{Q}^3
 \end{equation}
 and
  \begin{equation} \label{f below LG 2}
 f^*(Q) \leq  \mu_2 \left( 1 - \abs{Q}\right)^2 + 2\sigma \beta(Q) \abs{Q}^3 ,
 \end{equation}
 where $\sigma = \sigma(a, \, b,\, c)$, $\mu_i = \mu_i(a, \, b, \, c)$ for $i \in\{1, \, 2\}$ are explicitly computable positive constants. Moreover, setting $t := ac/b^2$ we have
 \begin{equation} \label{mu diverges}
 \frac{\sigma}{a}(t) = O(t^{-1/2}) \qquad \textrm{and} \qquad \frac{\mu_1}{a}(t) = O(1) \qquad \textrm{as } t\to + \infty  .
 \end{equation}
 \end{lemma}
 
\proof
Set $X := 1 - \abs{Q}$ and $Y := s_*\sqrt{2/3} - \abs{Q}$, for simplicity. We focus, at first, on the lower bound. The non-rescaled bulk potential $f$ satisfies the inequality
\begin{equation} \label{f below LG 11}
 f(Q) \geq Y^2 \left\{ \frac a2 + \frac{s_*^2c}{3} + \left(\frac{b}{3\sqrt 6} - c\sqrt{\frac23}s_* \right) Y + \frac c4 Y^2 \right\} +
 \frac{b}{6\sqrt 6} \beta(Q) \abs{Q}^3  ,
\end{equation}
which is a byproduct of the proof of \cite[Proposition 9]{MZ}. Thus, $f^*$ fulfills
\[
 f^*(Q) \geq X^2 \left\{\frac a2 + \frac{s_*^2c}{3} + \left(\frac{bs_*}{9} - \frac{2s_*^2c}{3} \right) X + \frac{s_*^2c}{6} X^2 \right\}
 + \frac{b s_*}{18} \beta(Q) \abs{Q}^3
\]
We set $\sigma := bs_*/18$. Now, we pay attention to the terms in brace: our goal is to minimize the mapping
\[
\phi\colon X \in [0, \, 1] \mapsto A + B X + C X^2   ,
\]
where we have set
\[
 A = \frac a2 + \frac{s_*^2c}{3}  , \qquad B = \frac{bs_*}{9} - \frac{2s_*^2c}{3}  , \qquad C = \frac{s_*^2c}{6}  ,
\]
in order to recover the lower estimate \eqref{f below LG 1} with $\mu_1 := \min_{[0, \, 1]} \phi$. It can be easily computed that $\phi$ attains its global minimum on $\R$ at the point
\[
-\frac{B}{2C} = 2 - \frac{b}{3s_* c} = 2 - \frac{4}{3\sqrt{1 + 24t} + 3}  ,
\]
and that $\mu_1 \geq \phi(-B/2C) > 0$, for all the possible values of $a, \, b, \, c$. We have proved the lower bound for~$f^*$.

We can argue analogously to bound $f^*$ from above. We consider the proof of \cite[Proposition 9]{MZ}, and use the inequality
\[
 \tr Q^3 = \pm \abs{Q}^3 \sqrt{\frac{1 - \beta(Q)}{6}} \leq \frac{\abs{Q}^3}{\sqrt6} \left(1 - \beta(Q) \right)
\]
in Equations (133)--(134). With the same calculations as in (135)--(138), we obtain
\[
 f(Q) \leq Y^2 \left\{ \frac a2 + \frac{cs_*^2}{3} + \left(\frac{b}{3\sqrt 6} - c\sqrt{\frac23}s_* \right) Y + \frac c4 Y^2 \right\} +
 \frac{b}{3\sqrt 6} \beta(Q) \abs{Q}^3  ,
\]
in place of the inequality \eqref{f below LG 11}. Thus, we can establish the upper estimate \eqref{f below LG 2}, with $\mu_2 := \max_{[0, \, 1]} \phi$.

Now, we investigate the behavior of $\mu_1$ and $\sigma$ as $t = ac/b^2 \to +\infty$. When $t$ is large enough, $-B/2C > 1$ and
\[
\mu_1 = \phi(1) = A + B + C  , 
\]
where the last expression can be explicitly calculated through simple algebra:
\[
 \mu_1 = \frac{a}{2} - \frac{s_*^2c}{6} + \frac{s_*b}{9} = \frac{36ac + b\sqrt{b^2 + 24ac} + b^2}{144c} = a \frac{36t + \sqrt{1 + 24 t} + 1}{144t}  .
\]
On the other hand,
\[
  \sigma = \frac{s_*b}{18} = \frac{b^2 + b \sqrt{b^2 + 24ac}}{72c} = a\frac{1 + \sqrt{1 + 24 t}}{72 t},
\]
and we conclude \eqref{mu diverges}.
\endproof

We see from the estimates \eqref{f below LG 1} and \eqref{f below LG 2} that $\mu_1, \, \mu_2$ and $\sigma$ can be understood as parameters governing the energy cost of uniaxiality and biaxiality, respectively. Furthermore, \eqref{mu diverges} suggests that biaxial solutions are energetically favorable when $t$ is large, that is, when $b$ is small, compared to $ac$. 

\begin{remark}
  In the limiting case $b = 0$, we can write $f$ as a function of $\abs{Q}$ only:
 \[
  f(Q) = k - \frac a2 \abs{Q}^2 + \frac c4 \abs{Q}^4  .
 \]
 The associated vacuum manifold is the unit sphere of $\Sz$, which we still denote by $S^4$. Since the latter is simply connected, the set $H^1_g(\Omega, \, S^4)$ is non empty for every choice of boundary datum. Thus, the solutions of Problem \eqref{energy} are the harmonic maps $\Omega \to S^4$. 
\end{remark}

The following lemma shows that almost uniaxial maps, just as uniaxial ones, must vanish at some point. This property will help us to obtain a lower bound for their energy.

\begin{lemma} \label{lemma: retract}
 If $Q\in C^1(\Omega, \, \Sz)$ is almost uniaxial with $\left.Q\right|_{\partial\Omega} = g$, and the boundary datum $g$ is not homotopically trivial, then $\min_{\overline\Omega} \abs{Q} = 0$.
\end{lemma}
\proof
We argue by contradiction, and assume that $s_0 := \sqrt{3/2}\min_{\overline\Omega} \abs{Q} > 0$. The almost uniaxiality hypothesis entails, in view of Remark~\ref{remark: beta lambda}, that 
$\lambda(Q(x))\neq 1/2$ for all $x\in\overline\Omega$. Since the image of $g$ lies in the vacuum manifold, by a connectedness argument we conclude that 
$r_0 := \max_{\overline\Omega} r(Q) < 1/2$. By \eqref{s abs}, the image of $Q$ is contained in the set
\[
 \NN_0 := \left\{ P \in \Sz \colon s(P) \geq s_0, \, r(P) \leq r_0 < 1/2 \right\}  .
\]
In particular, the boundary datum $g$ is homotopically trivial in $\NN_0$.

We claim that $\NN_0$ retracts by deformation on the vacuum manifold $\NN$. This entails the conclusion of the proof: composing $Q$ with the retraction yields a continuous extension of $g$ to a map $\Omega \to \NN$, which contradicts the nontriviality of $g$.

To construct a retraction, we exploit the representation formula of Lemma~\ref{lemma: representation}, and define the functions $K, \, H : \NN_0 \times [0, \, 1] \to \NN_0$ by
\[
K(P, \, t) := n^{\otimes 2} - \frac13 \Id + rt \left( m^{\otimes 2} - \frac13 \Id \right)  , \qquad 
 H(P, \, t) := \left\{ts + (1 - t)\sqrt{3/2}\right\} K (P, \, t)  .
\]
By the formulae \eqref{representation 2} and the continuity of the eigenvalues as functions of $P$, the mapping $P \mapsto (s(P), \, r(P))$ is well-defined and continuous on $\NN_0$. As
a consequence, $H$ is well-defined and continuous, if $K$ is. In addition, $H$ enjoys these properties: for all $P\in \NN_0$, we have $H(P, \, 1) = P$ and $H(P, \, 0)\in \NN$,
whereas $H(P, \, t) = P$ for all $(P, \, t)\in \NN\times [0, \, 1]$. 
It only remains to check that $K$ is well-defined and continuous.

Remark that each $P\in \NN_0$ has the leading eigenvalue of multiplicity one; in particular, $n = n(P)$ is uniquely determined, up to a sign, and $n^{\otimes 2}$ is well-defined. 
In case $r = r(P)\neq 0$, the second eigenvalue is simple as well, and the same remark applies to $m$. If $r(P) = 0$ then $K(P, \, t)$ is equally well-defined, 
regardless of the choice of $m$.

We argue somehow similarly for the continuity. If $\left\{(P_k, \, t_k)\right\}_{k\in\N}$ is a sequence in $\NN_0 \times [0, \, 1]$ converging to a fixed $(P, \, t)$, then
\begin{align*}
 \abs{K(P_k, \, t_k) - K(P, \, t)} \leq& \abs{n^{\otimes 2}(P_k) - n^{\otimes 2}(P)} + \abs{t_k - t}r(P_k) \abs{m^{\otimes 2}(P_k) - \frac13 \Id} \\
 &+ t \abs{r(P_k)m^{\otimes 2}(P_k) - r(P)m^{\otimes 2}(P)} + t \abs{r(P_k) - r(P)}  .
 \end{align*}
As the leading eigenvalue of $P\in  \NN_0$ is simple, standard results about the continuity of eigenvectors (see, for instance, \cite[Property 5.5, p.190]{QSS}) imply $n^{\otimes 2}(P_k) \to n^{\otimes 2}(P)$. 
If $r(P) = 0$, this is enough to conclude, since
\[
 \abs{K(P_k, \, t_k) - K(P, \, t)} \leq t r(P_k) \abs{m^{\otimes 2}(P_k)} + o(1) \to t r(Q) = 0
\]
as $k \to +\infty$. On the other hand, if $r(P) \neq 0$, then all the eigenvalues of $P$ are simple, and hence $m^{\otimes 2}(P_k) \to m^{\otimes 2}(P)$. 

Therefore, we can conclude that $K$ is continuous and $\NN_0$ retracts by deformation on $\NN$.
\endproof

The following proposition is the key element in the proof of Theorem~\ref{thm: intro-biaxial}. It is adapted from~\cite{Ch}, with minor changes. We report here the proof, for the reader's convenience.

\begin{prop} \label{prop: lower LG}
There exists a constant $M_1$ (independent on $\varepsilon$, $\lambda$ and $\mu_1$) such that, for all $Q\in C^1(\Omega, \, \Sz)$ fulfilling
\begin{equation} \label{hp lower LG}
 \min_{\overline\Omega}\abs{Q} = 0  , \qquad  \max_{\overline\Omega}\abs{Q} = 1 \qquad \textrm{and} \qquad \left. Q \right|_{\partial\Omega} = g,
\end{equation}
it holds that 
\[
 E_\varepsilon(Q) \geq \kappa_* \abs{\log\varepsilon} + \frac{\kappa_*}{2}\log\mu_1 - M_1  .
\]
\end{prop}
\proof
For $t > 0$, set
\[
 \Omega_t := \left\{x\in\Omega\colon \abs{Q(x)} > t\right\}  , \qquad \omega_t := \left\{x\in\Omega\colon \abs{Q(x)} < t\right\}  , \qquad
 \Gamma_t := \partial\Omega_t \setminus \partial\Omega = \partial \omega_t
\]
\[
 \Theta(t) := \int_{\Omega_t} \abs{\nabla \left( \frac{Q}{\abs{Q}}\right)}^2  , \qquad \nu(t) := \int_{\Gamma_t} \abs{\nabla\abs{Q}} \, \d\mathcal H^1  .
\]
Given a set $K \subset\subset\R^2$, let us define the radius of $K$ as
\[
 \textrm{rad}(K) := \inf\left\{ \sum_{i = 1}^n r_i \colon K \subseteq \bigcup_{i = 1}^n B(a_i, \, r_i)\right\}   .
\]

Having set these notations, we are ready to face the proof. It is well-known that $\nabla Q = 0$ a.e. in $\{Q = 0\}$, so we can write
\[
  \int_\Omega \abs{\nabla Q}^2 = \int_{\{\abs{Q} > 0\}} \abs{\nabla Q}^2 = \lim_{t\to 0^+} \int_{\Omega_t} \abs{\nabla Q}^2  ,
\]
by the monotone convergence theorem. This implies
\[
 \int_\Omega \abs{\nabla Q}^2 = \lim_{t\to 0^+} \int_{\Omega_t} \left\{ \abs{\nabla \abs{Q}}^2 + \abs{Q}^2 \abs{\nabla \left(\frac{Q}{\abs{Q}} \right)}^2\right\}
\]
and, applying the coarea formula, we deduce
\begin{equation} \label{lower LG 1}
 E_\varepsilon(Q) = \frac12 \int_0^1 \left\{ \int_{\Gamma_t}\left( \abs{\nabla \abs{Q}} + \frac{2 f^*(Q)}{\varepsilon^2 \abs{\nabla \abs{Q}}}\right) \d\mathcal H^1
 - 2t^2 \Theta'(t) \right\} \, \d t  .
\end{equation}
There is no trouble in dividing by $\abs{\nabla\abs{Q}}$ here. Indeed, combining \eqref{hp lower LG} and the Sard lemma we see that for a.e. $t\in (0, \, 1)$ $\Gamma_t$ is a non-empty, smooth curve in $\Omega$, and that $\abs{\nabla\abs{Q}} > 0$ on $\Gamma_t$. Of course, this implies $\nu(t) > 0$ for a.e. $t$ as well.

Let us estimate the terms in the right-hand side of \eqref{lower LG 1}, starting from the second one. Taking advantage of Lemma~\ref{lemma: f below LG} and of the H\"older inequality,
we obtain
\begin{equation} \label{lower LG 3}
\int_{\Gamma_t} \frac{2 f^*(Q)}{\varepsilon^2 \abs{\nabla \abs{Q}}} \geq \frac{2\mu_1(1 - t)^2}{\varepsilon^2} \int_{\Gamma_t} \frac{1}{\abs{\nabla \abs{Q}}} \, \d\mathcal H^1 
\geq \frac{2\mu_1(1 - t)^2 \mathcal H^1(\Gamma_t)^2}{\varepsilon^2 \nu(t)}  .
\end{equation}
Moreover, we have
\[
 \mathcal H^1(\Gamma_t) \geq 2 \textrm{diam}(\Gamma_t) \geq 4\textrm{rad}(\omega_t) \, ;
\]
to prove the latter inequality, remark that if $x, \, y\in \Gamma_t$ are such that $\abs{x - y} = \textrm{diam}(\Gamma_t)$, then $\omega_t$ is contained in the ball 
$B((x + y)/2, \, \abs{x - y}/2)$ and hence $\textrm{rad}(\omega_t) \leq \textrm{diam}(\Gamma_t)/2$. Combining this result with \eqref{lower LG 1} and \eqref{lower LG 3}, we find
\begin{equation} \label{lower GL 2}
\begin{aligned}
 E_\varepsilon(Q) &\geq \frac12 \int_0^1 \left\{ \nu(t) + \frac{32\mu_1(1 - t)^2 \textrm{rad}(\omega_t)^2}{\varepsilon^2 \nu(t)} \right\} \, \d t - \int_0^1 t^2 \Theta'(t) \, \d t \\
                  &\geq \int_0^1 \frac{4\sqrt 2}{\varepsilon} \mu_1^{1/2} \, (1 - t) \, \textrm{rad}(\omega_t) \, \d t - \int_0^1 t^2 \Theta'(t) \, \d t  ,
\end{aligned}
\end{equation}
where we have applied the inequality $A^2 + B^2 \geq 2AB$. Now, we pay attention to the last term, and we integrate it by parts. For all $\eta > 0$, we have
\[
 - \int_\eta^1 t^2 \Theta'(t) \, \d t = 2 \int_\eta^1 t \Theta(t) \, \d t + \eta^2 \Theta(\eta) \geq \int_\eta^1 t \Theta(t) \, \d t
\]
and, in the limit as $\eta \to 0$, by monotone convergence ($\Theta \geq 0$, $-\Theta'\geq 0$) we conclude
\[
 - \int_0^1 t^2 \Theta'(t) \, \d t \geq  2 \int_0^1 t \Theta(t) \, \d t  .
\]
Arguing as in \cite[Theorem 1]{Sand}, we can establish the bound
\[
 \Theta(t) \geq -\kappa_* \log\left(\textrm{rad}(\omega_t)\right) - C  ,
\]
where $C$ depends only on $\NN$ and the boundary datum (for further details, the reader might see~\cite{Sand} and the proof of Lemma~\ref{lemma: sand} in the sequel). Equation~\eqref{lower GL 2} implies
\[
 E_\varepsilon(Q) \geq  \int_0^1 \left\{ \frac{4\sqrt{2\mu_1}}{\varepsilon}  \, (1 - t) \, \textrm{rad}(\omega_t) - 2\kappa_*t \log\left(\textrm{rad}(\omega_t)\right) \right\} \, \d t - C  .
\]
Minimizing the function $r\in (0, \, +\infty)\mapsto 4\sqrt{2\mu_1}\varepsilon^{-1} (1 - t) r - 2\kappa_*t \log r$, we obtain the lower bound
\[
 E_\varepsilon(Q) \geq  \int_0^1 \left\{  2\kappa_*t -  2\kappa_*t \log\frac{\varepsilon \kappa_* t}{2\sqrt{2 \mu_1} (1 - t)} \right\} \, \d t - C  ,
\]
which implies
\[
  E_\varepsilon(Q) \geq  -2\kappa_* \int_0^1 \left\{ t \log\varepsilon - \frac t2 \log\mu_1 - t \log\frac{\kappa_* t}{2\sqrt{2} (1 - t)} \right\} \, \d t - C  .
\]
As $t \mapsto t \log\frac{\kappa_* t}{2\sqrt{2} (1 - t)}$ is integrable on $[0, \, 1]$, we can conclude the proof.
\endproof

\subsection{Proof of Theorem~\ref{thm: intro-biaxial}}

We are now in position to prove our main theorem about biaxiality. For $\varepsilon > 0$ small enough, we will construct a maximally biaxial function $\Pe\in H^1(\Omega, \, \Sz)$, which satisfies the estimate
\begin{equation} \label{biaxial 1}
 E_\varepsilon(\Pe) \leq \kappa_* \abs{\log\varepsilon} + \frac{\kappa_*}{2}\log\sigma + M_2
\end{equation}
for a constant $M_2$ independent of $\varepsilon$, $a$, $b$ and $c$. We claim that the Theorem follows from \eqref{biaxial 1}. Indeed, in view of \eqref{mu diverges}, we can find a number
$t_0 > 0$ such that
\[
 \frac{\kappa_*}{2} \log \frac{\mu_1}{a}(t) > \frac{\kappa_*}{2} \log \frac{\sigma}{a}(t) + M_1 + M_2
\]
holds, whenever $t \geq t_0$. This implies
\[
 \frac{\kappa_*}{2} \log\mu_1 - M_1 > \frac{\kappa_*}{2} \log\sigma + M_2
\]
and, combining this inequality with Proposition~\ref{prop: lower LG} and \eqref{biaxial 1}, we obtain that
\[
 E_\varepsilon(\Pe) < \inf_Q E_\varepsilon(Q)  ,
\]
the infimum being taken over functions $Q\in C^1(\Omega, \, \Sz)$ satisfying the conditions \eqref{hp lower LG}. In particular, the minimizers (which are smooth functions, by elliptic regularity) cannot satisfy~\eqref{hp lower LG}. Therefore, it must be
\[
 \min_{\overline Q} \abs{\Qe} > 0
\]
and, by applying Lemma~\ref{lemma: retract}, $\Qe$ is maximally biaxial. Thus, Theorem~\ref{thm: intro-biaxial} will be proved once the comparison function in \eqref{biaxial 1} is constructed.

At first, we will assume that the domain is a disk and the boundary datum of a special form. Next, we will extend our construction to the general setting.

\medskip
\emph{The case of a disk.} We assume that $\Omega$ is a disk in $\R^2$, of radius $R$, and that the boundary datum $c$ is given by \eqref{bordo geodesico}. 
Recall that, by Corollary~\ref{cor: bordo geodesico}, $c$ has minimal length, among all the non homotopically trivial loops in $\NN$, and that $\kappa_* = 3\pi/4$. Adopting polar coordinates on $\Omega$, for $(\rho, \, \theta)\in (0, \, R)\times[0, \, 2\pi]$ we set
\begin{equation} \label{biaxial 2}
 \Pe(\rho, \, \theta) := \sqrt{\frac32} \left\{ n_0(\theta)^{\otimes 2} - \frac13 \Id + r_\varepsilon (\rho)\left(m_0(\theta)^{\otimes 2} - \frac13 \Id\right)\right\}  ,                                                                                                                                                   
\end{equation}
where $m_0(\theta) := \left[-\sin(\theta/2), \, \cos(\theta/2), \, 0\right]^T$ and $r_\varepsilon$ is the continuous function
\[
 r_\varepsilon(\rho) := \begin{cases}
                        1 - \sigma^{1/2}\rho/\varepsilon  & \textrm{if } 0 \leq \rho \leq \sigma^{-1/2}\varepsilon \\
                        0                                 & \textrm{if } \sigma^{-1/2}\varepsilon \leq \rho \leq R  .\\
                       \end{cases}
\]
This construction is possible, provided that $\varepsilon < R\sigma^{1/2}$. Due to $r_\varepsilon(0) = 1$, function $\Pe$ can be extended continuously in the origin:
\[
 \lim_{\rho\to 0^+} \Pe(\rho, \, \theta) = - \sqrt{\frac32} \left\{ p_0^{\otimes 2} - \frac13 \Id \right\}  ,
\]
where $p_0 := \left[0, \, 0, \, 1\right]^T$. The derivatives of $\Pe$ can be computed explicitly: 
\[
 \abs{\frac{\partial\Pe}{\partial\rho}(\rho, \, \theta)}^2 = r_\varepsilon'^2(\rho)  , \qquad 
 \abs{\frac{\partial\Pe}{\partial\theta}(\rho, \, \theta)}^2 = \frac34 \left(1 - r_\varepsilon(\rho) \right)^2  ,
\]
and
\[
 \frac12\int_{\Omega} \abs{\nabla \Pe}^2 = \pi \int_0^R \left\{ \rho r_\varepsilon'^2(\rho) + \frac{3}{4\rho} \left(1 - r_\varepsilon(\rho) \right)^2 \right\} \d\rho 
 = \frac34 \pi \abs{\log\varepsilon} + \frac38 \pi \log\sigma + C .
\]
On the other hand, we exploit the upper estimate in Lemma~\ref{lemma: f below LG} to bound the potential energy:
\[
 \frac{1}{\varepsilon^2}\int_{\Omega} f^*(\Pe) \leq \frac{2\sigma}{\varepsilon^2} \textrm{meas}(B(0, \, \sigma^{-1/2}\varepsilon)) = 2\pi  .
\]
Therefore, we conclude that \eqref{biaxial 1} holds, in this case. 

\medskip
\emph{A general domain.} Now, consider an arbitrary domain $\Omega$, and fix a closed disk $B \subset\subset \Omega$. Since, by assumption, the boundary datum $g$ is non homotopically trivial, in view of Lemma~\ref{lemma: extension} it is possible to construct a smooth map $P\colon \Omega\setminus B \to \NN$, such that
\[
 \left.P\right|_{\partial\Omega} = g  \qquad \textrm{and} \qquad \left.P\right|_{\partial B} = c  .
\]
We define the map $\Pe\in H^1(\Omega, \, \Sz)$ by $\Pe(x) = P(x)$ if $x\in \Omega\setminus B$, and by the formula \eqref{biaxial 2} if $x\in B$.
The energy of $\Pe$, out of the ball $B$, is independent of $\varepsilon$ and $\sigma$, whereas the conclusion of Step 1 provides a bound for the energy on $B$.
Hence, \eqref{biaxial 1} follows.

\begin{remark}
Since this argument does not provide an explicit lower bound for $\abs{\Qe}$, we cannot infer that singularity profiles are bounded away from zero (for the definition of singularity profiles, see the Introduction).
\end{remark}

\section{Asymptotic analysis of the minimizers}
\label{sect: convergence}

This section investigates the behavior of minimizers of \eqref{energy} as $\varepsilon\searrow 0$, and contains the proofs of Proposition~\ref{prop: intro-defects} and Theorem~\ref{th: intro-convergence}. We start by recalling some well-known properties of minimizers.

\begin{lemma} \label{lemma: L infty}
If $\ue $ is a minimizer for Problem \eqref{energy}, then 
\[
\norm{\ue }_{L^\infty(\Omega)} \leq 1 \quad \textrm{and} \quad \norm{\nabla\ue }_{L^\infty(\Omega)} \leq \frac{C}{\varepsilon}  .
\]
\end{lemma}
\proof
The $L^\infty$ bound on $\ue $ can be easily established via a comparison argument. Assume, by contradiction, that $\abs{\ue(x_0)} > 1$ for some $x_0\in\Omega$, and define
\[
v_\varepsilon(x) := \begin{cases}
\ue(x) & \textrm{if } \abs{\ue(x)}\leq 1 \\ 
\displaystyle\frac{\ue(x)}{\abs{\ue(x)}} & \textrm{otherwise}  .
\end{cases}
\]
Clearly $\abs{\nabla v_\varepsilon}\leq \abs{\nabla\ue}$ and, by \ref{hp: growth}, $f(v_\varepsilon) \leq f(\ue)$, with strict equality at least at the point $x_0$. Thus, we infer
$E_\varepsilon(v_\varepsilon) < E_\varepsilon(\ue)$, which contradicts the minimality of $\ue$. The estimate on the gradient can be deduced from Equation \eqref{EL}, with the help of the previous bound and \cite[Lemma A.2]{BBH2}.
\endproof

\begin{lemma}[Pohozaev identity] \label{lemma: pohoz}
Let $G\subseteq \R^2$ be any subdomain of $\Omega$, and $x_0\in G$. Denote by $\nu$ the unit external normal to $\partial G$ and by $\tau$ the unit tangent to $\partial G$, oriented so that 
$(\tau, \, \nu)$ is direct. Then, any solution $\ue$ of Equation \eqref{EL} satisfies
\begin{equation}
\begin{aligned}
 \frac{1}{\varepsilon^2}\int_G f(\ue )  &+ \frac12 \int_{\partial G} (x - x_0)\cdot\nu \abs{\frac{\partial\ue }{\partial\nu}}^2 \d\mathcal H^1 \\ 
&= \int_{\partial G} \left\{\frac12 (x - x_0)\cdot \nu \abs{\frac{\partial \ue }{\partial\tau}}^2 - (x - x_0)\cdot\tau \frac{\partial\ue }{\partial\nu}:\frac{\partial \ue }{\partial\tau} +
(x - x_0)\cdot\nu\frac{1}{2\varepsilon^2}f(\ue )\right\} \d\mathcal H^1  .
\end{aligned}
\label{pohoz}
\end{equation}
\end{lemma}

\proof
The lemma can be proved arguing exactly as in \cite[Theorem III.2]{BBH}.
\endproof

\begin{remark}
When considering the Landau-de Gennes equation \eqref{LG EL}, the additional term $b(\tr\Qe^2)\Id/3$ does not play any role in the proof of the Pohozaev identity. 
Indeed, assume~$x_0 = 0$, multiply both sides of \eqref{LG EL} by $x_k Q_{ij, \, k}$, sum over $i, \, j, \, k$ and integrate over $G$. We obtain
\[
-\int_G Q_{\varepsilon, \, ij, \, ll} \, x_k Q_{\varepsilon, \, ij, \, k} \ \d x + 
\frac{1}{2\varepsilon^2}\int_G \frac{\partial f(\Qe)}{\partial Q_{\varepsilon, \, ij}} \, Q_{\varepsilon, \, ij, \, k} \, x_k \ \d x + 
\frac13 b \int_G x_k Q_{\varepsilon, \, ii, \, k} \tr \Qe^2 \ \d x = 0  ,
\]
and the third integral vanishes, since $Q_{\varepsilon, \, ii} = 0$. The proof follows exactly as in the previous case; the reader is referred to \cite[Lemma 2]{MZ} for more details.
\end{remark}

\begin{lemma} \label{lemma: energy estimate}
Let $\Omega \subseteq \R^2$ and let $\ue$ be a minimizer for Problem \eqref{energy}. Then, there exists a constant $C$, depending on $\Omega$ and $G$, such that 
\[
E_\varepsilon(\ue) \leq \kappa_*\abs{\log\varepsilon} + C  .
\]
\end{lemma}

\proof
The proof of this lemma is a slightly different version of \cite[Theorem III.1]{BBH} (see also \cite{Ch}).
Let $ (\eta_1, \, \eta_2, \, \ldots, \, \eta_m)\in \Gamma(\NN)^m$ be an $m$-uple which achieve the minimum in \eqref{kappa *}, that is,
\begin{equation} \label{quasi min}
 \prod_{j = 1}^m \eta_j \cap \prod_{i = 1}^k\gamma_i \neq \emptyset  , \qquad  \frac{1}{4\pi}\sum_{j = 1}^m \lambda(\eta_j)^2 = \kappa_*  ,
\end{equation}
and for each $j \in \{ 1, \, \ldots, \, m\}$ choose a loop $b_j\in\eta_j$, of minimal length (i.e., $\lambda(\eta_j) = 2\pi |b'_j|$). 
Let $B_1, \, \ldots, \, B_m$ be mutually disjoint, closed disks in $\Omega$, of radius $r$. Applying Lemma~\ref{lemma: extension}, we find a smooth function $v\colon \Omega\setminus\cup_{j = 1}^m B_j \to \NN$, such that $v = g$ on $\partial \Omega$ and $v = b_j$ on $\partial B_j$, for each $j$.

We extend $v$ to a function $v_\varepsilon\colon \overline \Omega \to \NN$ in the following way. On $\Omega\setminus\cup_{j = 1}^m B_j$, set $v = v_\varepsilon$, whereas 
on each ball $B_j$, denoting by $(\rho, \, \theta)$ the polar coordinates around the center $x_j$ of the ball, set
\[
v_\varepsilon(x) := \begin{cases}
\varepsilon^{-1}\rho \, b_j \left(x_j + r e^{i\theta}\right) & \textrm{if } 0 < \rho < \varepsilon \\
                        b_j \left(x_j + r e^{i\theta}\right) & \textrm{if } \varepsilon \leq \rho < r  .
\end{cases}
\]
Since $v_\varepsilon\in H^1_g(\Omega, \, \R^k)$, the minimality of $\ue$ entails $E_\varepsilon(\ue)\leq E_\varepsilon(v_\varepsilon)$.
Computing $E_\varepsilon(v_\varepsilon)$ will be enough to conclude \eqref{kappa *}. We find
\[
E_\varepsilon(v_\varepsilon; \, \Omega\setminus \bigcup_{j = 1}^m B_j) = \frac12 \int_{\Omega\setminus \bigcup_{j = 1}^m B_j}\abs{\nabla v}^2 = C  ,
\]
\[
E_\varepsilon(v_\varepsilon; \, B_j) \leq \int_{B_j} C\varepsilon^{-2} \leq C  ,
\]
and, passing to polar coordinates,
\[
E_\varepsilon(v_\varepsilon;  B_j \setminus B_\varepsilon(x_j)) = \frac12 \int_\varepsilon^r \frac{\d\rho}{\rho} \int_{S^1} d\omega \, \abs{b_j'(\omega)}^2 \leq 
\frac{1}{4\pi} \lambda(\eta_j)^2 \abs{\log\varepsilon} + C  .
\]
Combining these bounds, with the help of \eqref{quasi min} we conclude.
\endproof

\subsection{Localizing the singularities}
\label{subsect: localizing singularities}

In this subsection, we will prove Proposition~\ref{prop: intro-defects}. Namely, we will show that the image of $\ue$ lies close to the vacuum manifold, except on the union of a finite number of small balls. Analogous results have been established for the Ginzburg-Landau model in \cite{BBH}, in case the domain $\Omega\subseteq\R^2$ is star-shaped. This technical assumption has been removed in \cite{BR} and \cite{Str} (see also \cite{Cetraro} for more details).

We introduce a (small) parameter $0 < \alpha \leq 1$, whose value is going to be adjusted later, and we set
\[
 \ee(\ue) = \frac12 \abs{\nabla\ue}^2 + \frac{1}{\varepsilon^2} f(\ue)  .
\]
We claim the following

\begin{prop} \label{prop: palle}
Let $\delta\in (0, \, \delta_0)$ be fixed. For all $\varepsilon > 0$, there exists a finite set $X_\varepsilon = \left\{x_1, \, \ldots, \, x_k\right\} \subset \Omega$, whose cardinality is bounded independently of $\varepsilon$, such that
\begin{equation} \label{palle}
\dist(\ue(x), \, \NN)\leq \delta \qquad \textrm{if } \dist(x, \, X_\varepsilon) > \lambda_0 \varepsilon  ,
\end{equation}
where $\lambda_0 > 0$ is a constant independent of $\varepsilon$, and
\begin{equation} \label{palle bis}
\varepsilon^{4\alpha} \ee(\ue) (x) \leq C_\alpha \qquad \textrm{if } \dist(x, \, X_\varepsilon) > \varepsilon^\alpha  .
\end{equation}
\end{prop}

Proposition~\ref{prop: palle} clearly implies Proposition~\ref{prop: intro-defects}.
As a first step in the proof, we show that $\ee(\ue)$ solves an elliptic inequality, in the regions where $\ue$ lies close to the vacuum manifold.

\begin{lemma} \label{lemma: bochner}
 Assume $\omega \subset \Omega$ is an open set, such that $\dist(\ue(x), \, \NN) \leq \delta$ holds for all $x\in\omega$ and all $\varepsilon > 0$. Then, $\ee(\ue)$ solves pointwise in $\omega$ the inequality
 \[
 - \Delta \ee(\ue) \leq C \ee^2 (\ue)  .
 \]
\end{lemma}
\proof
Reminding that $\ue$ is a solution of Equation \eqref{EL}, we compute plainly

\begin{align*}
-\frac12 \Delta \abs{\nabla \ue}^2 &= - \nabla \left( \Delta\ue\right) \cdot \nabla\ue - \abs{\nabla^2\ue}^2 
= - \frac{1}{\varepsilon^2} \nabla\ue : D^2 f(\ue) \nabla\ue - \abs{\nabla^2\ue}^2  ,
\end{align*}
where $\abs{\nabla^2\ue}^2 = \sum_{i, \, j}\abs{\partial_i\partial_j \ue}^2$, and
\begin{align*}
-\frac{1}{\varepsilon^2} \Delta f(\ue) &= - \frac{1}{\varepsilon^2} \nabla \left(Df(\ue) \right) \cdot \nabla\ue -  \frac{1}{\varepsilon^2} Df(\ue)\cdot \Delta \ue \\
&= - \frac{1}{\varepsilon^2} \nabla\ue : D^2 f(\ue) \nabla\ue - \frac{1}{\varepsilon^4} \abs{Df(\ue)}^2  .
\end{align*}

Adding these contributions, we obtain
\begin{equation} \label{boch 1}
 -\Delta\ee(\ue) + \abs{\nabla^2\ue}^2 + \frac{1}{\varepsilon^4} \abs{Df(\ue)}^2 = - \frac{2}{\varepsilon^2} \nabla\ue : D^2 f(\ue) \nabla\ue  .
\end{equation}
Hypothesis \ref{hp: df} provides
\begin{equation} \label{boch 2}
 \frac{1}{\varepsilon^4}\abs{D f(\ue)}^2 \geq \frac{m_0}{\varepsilon^4} \dist^2(\ue, \, \NN)  .
\end{equation}
Moreover, the image $\ue(\omega)$ lies close to $\NN$ by assumption, so the right-hand side of \eqref{boch 1} can be estimated by the local Lipschitz continuity of $D^2f$:
\begin{align*}
 - \frac{2}{\varepsilon^2} \nabla\ue : D^2 f(\ue) \nabla\ue &\leq - \frac{2}{\varepsilon^2} \nabla\ue : D^2 f(\pi(\ue)) \nabla\ue + \frac{2}{\varepsilon^2} \abs{D^2f(\ue) - D^2f(\pi(\ue))} \abs{\nabla \ue}^2 \\
 &\leq - \frac{2}{\varepsilon^2} \nabla\ue : D^2 f(\pi(\ue)) \nabla\ue + \frac{C}{\varepsilon^2}\dist(\ue, \, \NN) \abs{\nabla \ue}^2 \\
 & \leq \frac{C}{\varepsilon^2}\dist(\ue, \, \NN) \abs{\nabla \ue}^2  .
 \end{align*}
For the latter inequality, remind that every point $p\in \NN$ is a minimizer for $f$, so $D^2 f(p) \geq 0$. We infer 
\[
 - \frac{2}{\varepsilon^2} \nabla\ue : D^2 f(\ue) \nabla\ue \leq \frac{m_0}{\varepsilon^4} \dist^2(\ue, \, \NN) + C \abs{\nabla\ue}^4
\]
and the first term can be reabsorbed in the left-hand side of \eqref{boch 1}, by means of \eqref{boch 2}. This concludes the proof.
\endproof

Our next ingredient is a Clearing Out lemma, which relies crucially on \ref{hp: df}.

\begin{prop}[Clearing Out] \label{prop: clearing out}
There exist some positive constants $\lambda_0$ and $\mu_0$ with the following property: for all $x_0\in\Omega$ and all $l\in[\lambda_0\varepsilon, \, 1]$, if the minimizer $\ue$ satisfies
\begin{equation}
\frac{1}{\varepsilon^2}\int_{B(x_0, \, 2l)\cap\Omega} f(u_\varepsilon) \leq \mu_0
\label{small f}
\end{equation}
then
\begin{equation}
\dist(u_\varepsilon(x), \, \NN) \leq \delta\qquad \textrm{for all } x\in \Omega\cap B(x_0, \, l)  .
\label{small dist}
\end{equation}
\end{prop}

\proof Set
\[
 f_0 := \min\left\{ f(v) \colon \dist(v, \, \NN) \geq \delta  , \: \abs{v} \leq 1 \right\}  ,
\]
and remark that $f_0 > 0$, because it is the minimum of a strictly positive function on a compact set. We define
\begin{equation}
\lambda_0 := \frac{\delta}{2C}  , \qquad \mu_0 := \frac\pi 2 \lambda_0^2\min\left\{ f_0, \, \frac18 m_0\delta^2 \right\} ,
\label{lambda_0}
\end{equation}
where $C$ is a constant such that $\abs{\nabla\ue}\leq C\varepsilon^{-1}$ (such a constant exists, by Lemma~\ref{lemma: L infty}).
We are going to check that this choice of $\lambda_0, \, \mu_0$ works. To do so, we proceed by contradiction and assume there is some point $x\in B(x_0, \, l)$ such that
$\dist(\ue(x), \, \NN) > \delta$. Firstly, we remark that this assumption implies $\dist(x, \, \partial\Omega) > \lambda_0\varepsilon$. Indeed, if it were $\dist(x, \, \partial\Omega) \leq 
\lambda_0\varepsilon$ then, in view of~\ref{hp: g}, we would have
\[
\dist(\ue(x), \, \NN) \leq \norm{\nabla\ue}_{L^\infty} \dist(x, \, \partial\Omega) \leq C\lambda_0 = \frac{\delta}{2}  .
\]
It follows that the ball $B(x, \, \lambda_0\varepsilon)$ is entirely contained in $\Omega\cap B(x_0, \, 2l)$. In addition, for all $y\in B(x, \, \lambda_0\varepsilon)$ we have
\[
\dist(\ue(y), \, \NN) \geq \dist(\ue(x), \, \NN) - \abs{\ue(x) - \ue(y)} > \delta - \lambda_0 \varepsilon  \norm{\nabla\ue}_{L^\infty} \geq \frac{\delta}{2}  .
\]
Due to Remark~\ref{remark: f} and Lemma~\ref{lemma: L infty}, this implies
\[
\frac{1}{\varepsilon^2}\int_{\Omega\cap B(x_0, \, 2l)} f(\ue) \geq \frac{1}{\varepsilon^2}\int_{B(x, \, \lambda_0\varepsilon)} f(\ue) \geq 
\pi \lambda_0^2 \min\left\{f_0, \, \frac18 m_0\delta^2\right\} = 2 \mu_0  ,
\]
which contradicts the hypothesis \eqref{small f}.
\endproof

The two following results can be found in \cite[Section IV.5]{Cetraro}. The proofs carry over to our setting, without any change.

\begin{lemma} \label{lemma: local estimate}
Let $x_0\in \Omega$. There exists a constant $C_\alpha$, depending only on $\alpha$, $g$ and $\Omega$, such that 
\[
\frac{1}{\varepsilon^2} \int_{B(x_0, \, \varepsilon^\alpha)\cap\Omega} f(\ue) \leq C_\alpha   .
\]
\end{lemma}

\begin{prop} \label{prop: concentration} There exists a constant $\eta_\alpha  > 0$, independent of $\varepsilon$, with the following property: if a point $x_0\in\Omega$ verifies
\begin{equation}
\int_{B(x_0, \, 2\varepsilon^\alpha)\cap\Omega} \abs{\nabla \ue}^2 \leq \eta_\alpha \abs{\log\varepsilon} + C
\label{concentration}
\end{equation}
then $\dist(\ue(x), \, \NN) \leq \delta$ for all $x\in B(x_0, \, \varepsilon^\alpha)$.
\end{prop}

Proposition~\ref{prop: concentration} provides a concentration result for the energy, which will be crucial in our argument.
Reducing, if necessary, the value of $\eta_\alpha$, we are able to show another estimate for minimizers satisfying \eqref{concentration}. This will be the final ingredient in our proof of Proposition~\ref{prop: palle}.

\begin{prop} \label{prop: chen struwe}
There exist constants $\eta_\alpha, \, C_\alpha > 0$ (with $C_\alpha$ depnding only on $\alpha, \, \eta_\alpha$) such that, if $\ue$ verifies the condition \eqref{concentration} for some $x_0\in\Omega$, then
\[
 \varepsilon^{4\alpha} \ee(\ue)(x_0) \leq C_\alpha  .
\]
\end{prop}
\proof

We suppose, at first, that $B(x_0, \, \varepsilon^\alpha)\subset \Omega$.
In view of Proposition~\ref{prop: concentration}, we can assume that $\dist(\ue, \, \NN)\leq \delta$ on $B(x_0, \, \varepsilon^{\alpha})$. 
Furthermore, \eqref{concentration} and Lemma~\ref{lemma: local estimate} provide
\begin{equation}
 E_\varepsilon(\ue, \,  B(x_0, \, \varepsilon^\alpha)\cap\Omega) \leq \eta_\alpha \abs{\log\varepsilon} + C_\alpha  .
\label{concentration 1}
\end{equation}

We claim that there exists a radius $r\in (\varepsilon^{2\alpha}, \, \varepsilon^\alpha)$ such that 
\begin{equation}
\int_{\partial B(x_0, \, r)\cap\Omega} \left\{ \abs{\nabla \ue}^2  + \frac{1}{\varepsilon^2}f(\ue) \right\} \leq \frac{2\eta_\alpha}{\alpha r}  .
\label{concentration 2}
\end{equation}
Indeed, if \eqref{concentration 2} were false, integrating over $(\varepsilon^{2\alpha}, \, \varepsilon^\alpha)$ we would obtain
\[
  E_\varepsilon(\ue, \,  B(x_0, \, \varepsilon^\alpha)\cap\Omega) \geq \frac{2\eta_\alpha}{\alpha} \int_{\varepsilon^{2\alpha}}^{\varepsilon^\alpha} \frac{\d r}{r} = 2\eta_\alpha \abs{\log\varepsilon}  ,
\]
which contradicts \eqref{concentration 1}, for $\varepsilon\ll 1$. Thus, the claim is established.

Set
\[
 \epsilon := \varepsilon/r \qquad \textrm{and} \qquad v_\epsilon (x) := \ue\left(\frac{x + x_0}{r}\right) \qquad \textrm{for } x\in B := B(0, \, 1)  .
\]
As a consequence of the scaling, we deduce $e_\epsilon(v_\epsilon) = r^2 \ee(\ue)$, hence
\[
 E_\epsilon(v_\epsilon, \, B) = E_\varepsilon(\ue, \, B(x_0, \, r))
\]
and $v_\epsilon$ minimizes the energy $E_\epsilon$ among the maps $w\in H^1(B, \, \R^d)$, with $\left. w\right|_{\partial B} = \left. v_\epsilon\right|_{\partial B}$.
Moreover, \eqref{concentration 2} transforms into
\begin{equation} \label{stima al bordo} 
 \int_{\partial B} e_\epsilon(v_\epsilon) \leq \frac{2\eta_\alpha}{\alpha}  .
\end{equation}
We will take advantage of the following property, whose proof is postponed.

\begin{lemma} \label{lemma: comparison}
The energy of $v_\epsilon$ is controlled by $\eta_\alpha$, that is,
\[
 \int_B e_\epsilon(v_\epsilon) \leq C_\alpha\eta_\alpha  .
\]
\end{lemma}

Recall also that, due to Lemma~\ref{lemma: bochner}, $e_\epsilon(v_\epsilon)$ solves an elliptic inequality. 
Thus, we are in position to invoke a result by Chen and Struwe (\cite{ChStr} --- the reader is also referred to \cite[Theorem 2.2]{Sch}):
provided that $\eta_\alpha$ is small enough, Lemma \ref{lemma: comparison} implies the estimate
\[
 r^2 e_\varepsilon(\ue) (x_0) = e_\epsilon(v_\epsilon) (x) \leq \int_B e_\epsilon(v_\epsilon) \leq C_\alpha \eta_\alpha  .
\]
This concludes the proof, in case $B(x_0, \, \varepsilon^\alpha)$ does not intersect the boundary. 

We still have to cover the case $B(x_0, \, \varepsilon^\alpha)\nsubseteq\Omega$, but this entail no significant change in the proof (nor in the proof of Lemma~\ref{lemma: comparison}). As we deal with a local result, we can straighten the boundary and assume that $\Omega$ coincides locally with the set $\R^n_+$. In place of the Chen-Struwe result we can exploit 
\cite[Theorem 2.6]{Sch}, which deals with the Dirichlet boundary condition.
\endproof

\begin{proof}[Proof of Lemma~\ref{lemma: comparison}]
We split the proof in steps, for clarity.

\setcounter{step}{0}
\begin{step}[Construction of the harmonic extension]
The composition $\pi(v_\epsilon)$ is well defined, since the image of $v_\epsilon$ lies close to the vacuum manifold. Set $\sigma_\epsilon := \dist(v_\epsilon, \, \NN) = \abs{v_\epsilon - 
\pi(v_\epsilon)}$, and denote by $\omega_\epsilon$ an harmonic extension of $\left.\pi(v_\epsilon)\right|_{\partial B}$ on $B$. The existence of such an extension is a classical result by Morrey (see, for instance, \cite{Morrey}). Lemma 4.2 in \cite{SU} and \eqref{lemma: comparison} imply that
 \begin{equation} \label{energia armonica}
  \int_B \abs{\nabla \omega_\epsilon}^2 \leq C \int_{\partial B} e_\epsilon(v_\epsilon) \leq C_\alpha \eta_\alpha  .
 \end{equation}
 We wish to use $\omega_\epsilon$ as a comparison map, in order to obtain the $H^1$ bound for $v_\epsilon$; to do so, we have to take care of the boundary condition on $\partial D$. 
 \end{step}
 
 \begin{step}[An auxiliary map]
 It will be useful to introduce an auxiliary map $\varphi_\epsilon$. Using polar coordinates on $D$, we define $\varphi_\epsilon$ by the formula
 \[
  \varphi_\epsilon(\rho, \, \theta) = \begin{cases}
                               \epsilon^{-1} (\rho - 1 + \epsilon) \sigma_\epsilon(\theta) & \textrm{if } 1 - \epsilon \leq \rho \leq 1 \\
                               0                                                           & \textrm{if } \rho < 1 - \epsilon  ,
                              \end{cases}
 \]
 We claim that
 \begin{equation} \label{bound alpha_n}
  \int_B \left\{ \abs{\nabla \varphi_\epsilon}^2 + \frac{1}{\epsilon^2}\varphi_\epsilon^2 \right\} \leq C \epsilon
 \end{equation}
 and check it by a straightforward computation. Indeed,
\begin{align*}
  \int_B \left\{ \abs{\nabla \varphi_\epsilon}^2 + \frac{1}{\epsilon^2}\varphi_\epsilon^2 \right\} \leq 
  \int_0^{2\pi} \d\theta \int_{1 - \epsilon}^1 \d\rho \left\{\frac{1}{\epsilon ^2}\rho \abs{\sigma_\epsilon(\theta)}^2 +  
  \frac{1}{\rho} \abs{\sigma_\epsilon'(\theta)}^2 + \rho \abs{\sigma_\epsilon(\theta)}^2 \right\} \\
  \leq \frac{1}{\epsilon^2} \left(\epsilon - \frac{\epsilon^2}{2}\right) \norm{\sigma_\epsilon}_{L^2(S^1)}^2 - \log(1 - \epsilon) \norm{\sigma_\epsilon'}_{L^2(S^1)}^2 + 
   \left(\epsilon - \frac{\epsilon^2}{2}\right) \norm{\sigma_\epsilon}_{L^2(S^1)}^2
\end{align*}
 and \eqref{bound alpha_n} follows from \eqref{stima al bordo}, since $\abs{\sigma'_\epsilon} \leq \abs{\nabla(v_\epsilon - \pi(v_\epsilon))}$ and $\sigma_\epsilon^2 \leq C f(v_\epsilon)$.
 \end{step}
 
 \begin{step}[Construction of a normal field on $\NN$]
 By construction, $\left. (v_\epsilon - \omega_\epsilon) \right|_{\partial B}$ is a normal field on $\NN$, whose modulus is given by 
 $\sigma_\epsilon$. We want to extend it to a map $\nu_\epsilon\colon B \to \R^d$, so that $\nu_\epsilon(x)$ is orthogonal to $\NN$ at the point $\omega_\epsilon(x)$ and $\abs{\nu_\epsilon(x)} = \varphi_\epsilon(x)$, for all $x\in D$. At first, one may work locally, near a point $x_0\in \partial \Omega$, and exploit the existence of an orthonormal frame of normal vectors, defined on some neighborhood of $\omega_\epsilon(x_0)$. Then, the construction of $\nu_\epsilon$ is completed by a partition of the unity argument.
 \end{step}
 
 \begin{step}[Construction of a comparison map]
 Set $\widetilde\omega_\epsilon := \omega_\epsilon + \nu_\epsilon$. It follows from the previous steps that $\widetilde\omega_\epsilon$ enjoys these properties:
 \[
  \left. \widetilde\omega_\epsilon \right|_{\partial B} = \left. v_\epsilon \right|_{\partial B}  ,
 \]
 \[\pi(\widetilde\omega_\epsilon(x)) = v_\epsilon(x) \quad \textrm{and} \quad \dist(\widetilde\omega_\epsilon(x), \, \NN) = \varphi_\epsilon(x) \quad \textrm{for all } x\in B .
 \]
 In particular, $\widetilde\omega_\epsilon$ is an admissible comparison map for $v_\epsilon$. By this information and Lemma~\ref{lemma: npp}, we infer a bound for the gradient of $\widetilde\omega_\epsilon$:
 \[
  \abs{\nabla \widetilde\omega_\epsilon}^2 \leq (1 + C \varphi_\epsilon) \abs{\nabla \omega_\epsilon}^2 + C\left( \abs{\nabla \varphi_\epsilon}^2 + \varphi_\epsilon^2 \right)
 \]
 Since $\abs{\varphi_\epsilon} \leq \delta$ by construction, integrating this inequality over $D$ and exploiting \eqref{stima al bordo} we obtain
 \begin{equation} \label{bound grad v_n}
  \norm{\nabla \widetilde\omega_\epsilon}^2_{L^2(D)} \leq (1 + \delta) \norm{\nabla\omega_\epsilon}^2_{L^2(D)} + C \epsilon  .
 \end{equation}
 The potential energy of $\widetilde\omega_\epsilon$ is estimated by means of Remark~\ref{remark: f}:
 \[
  \frac{1}{\epsilon^2} \int_B f(\widetilde\omega_\epsilon) \leq \frac{M_0}{2\epsilon^2} \int_B \varphi^2_\epsilon  .
 \]
 Combining this inequality with \eqref{bound alpha_n} and \eqref{bound grad v_n}, we deduce
 \begin{equation} \label{bound v_n}
 \int_B e_\epsilon(v_\epsilon) \leq \int_B e(\widetilde\omega_\epsilon) \leq (1 + \delta) \norm{\nabla\omega_\epsilon}^2_{L^2(D)} + C \epsilon  .
 \end{equation}
 With this estimate and \eqref{energia armonica}, we complete the proof. \qed
 \end{step}
 \let\qed\relax
 \end{proof}

Having established all these preliminary results, Proposition~\ref{prop: palle} follows easily from a covering argument as, for instance, the one in \cite{Cetraro} (see also \cite[Chapter IV]{BBH}).
 
\begin{proof}[Proof of Proposition~\ref{prop: palle}]
By Vitali covering lemma, we can find a finite family of points $\{y_i\}_{i\in I}$ such that
\[
\Omega \subseteq \bigcup_{i\in I} B(y_i, \, 3\varepsilon^\alpha) 
\]
and
\[
 B(y_i, \, \varepsilon^\alpha)  \cap  B(y_j, \, \varepsilon^\alpha)  = \emptyset \qquad \textrm{if } i\neq j  .
\]
Let $\eta_\alpha = \eta_\alpha(\alpha, \, \delta)$ be given by Proposition~\ref{prop: chen struwe}. Define $J_\varepsilon$ as the subset of indexes $i\in I$ for which the inequality 
\[
\int_{B(y_i, \, 3\varepsilon^\alpha) } \abs{\nabla\ue}^2 \geq \eta_\alpha \left( \abs{\log\varepsilon} + 1 \right)
\]
holds; then, by Lemma~\ref{lemma: energy estimate}, we have
\begin{equation}
\eta_\alpha \left( \abs{\log\varepsilon} + 1 \right) \textrm{card} (J_\varepsilon) \leq \sum_{j\in J_\varepsilon} \int_{ B(y_i, \, 3\varepsilon^\alpha) }\abs{\nabla\ue}^2 \leq 
C \left(\abs{\log\varepsilon} + 1 \right)
\label{palle 1}
\end{equation}
(to prove the last inequality, recall that there is a universal constant $C$ such that each point of $\Omega$ is covered by at most $C$ balls of radius $3\varepsilon^\alpha$).
It follows that $\textrm{card}(J_\varepsilon)$ is bounded independently of $\varepsilon$. Moreover, Propositions~\ref{prop: concentration} and~\ref{prop: chen struwe} imply that
\[
\dist(\ue(x), \, \NN)\leq\delta \qquad \textrm{if } x\in  B(y_i, \, 3\varepsilon^\alpha) \textrm{ and } i\in I \setminus J_\varepsilon
\]
\[
\varepsilon^{4\alpha}\ee(\ue)(x) \leq C_\alpha \qquad \textrm{if } x\in  B(y_i, \, \varepsilon^\alpha) \textrm{ and } i\in I \setminus J_\varepsilon  .
\]

Now, let us fix an index $i\in J_\varepsilon$, and let us focus on $B(y_i, \, 3\varepsilon^\alpha)$. Being $\lambda_0 = \lambda_0(\delta)$ and \mbox{$\mu_0 = \mu_0(\delta)$}
given by Proposition~\ref{prop: clearing out}, we consider a finite covering $\{B(x_m^i, \, 3\lambda_0\varepsilon)\colon m\in\Lambda_{\varepsilon, \, i}\}$ of $B(y_i, \, 3\varepsilon^\alpha)$,
such that
\[
B(x_m^i, \, \lambda_0\varepsilon) \cap B(x_n^i, \, \lambda_0\varepsilon)= \emptyset \qquad \textrm{if } m\neq n  ,
\]
and we define the set $L_{\varepsilon, \, i}$ of indexes $m\in \Lambda_{\varepsilon, \, i}$ such that
\[
\frac{1}{\varepsilon^2}\int_{B(x_m^i, \, 3\lambda_0\varepsilon)} f(\ue) > \mu_0  .
\]
Since $\varepsilon^{-2}\int_{B(x_i, \, 3\varepsilon^\alpha)} f(\ue)$ is controlled by Lemma~\ref{lemma: local estimate}, we can bound the cardinality of $L_{\varepsilon, \, i}$, independently 
of~$\varepsilon$, exactly as in \eqref{palle 1}. By Proposition~\ref{prop: clearing out}, we have that $\dist(\ue(x), \, \NN) \leq \delta$ if $x\in B(x_m^i, \, 3\lambda_0\varepsilon)$ and~$m\notin L_{\varepsilon, \, i}$.

Combining all these facts, we conclude easily.
\end{proof}


Denote by $x_1^\varepsilon, \, x_2^\varepsilon, \ldots, \, x_{k_\varepsilon}^\varepsilon$ the elements of $X_\varepsilon$. For any given sequence $\varepsilon_n\searrow 0$ we can extract a renamed subsequence, such that $k_{\varepsilon_n}$ is independent of $n$ (say, $k_{\varepsilon_n} = N'$) and
\[
x_i^{\varepsilon_n} \to L_i \qquad \textrm{for } \: i \in \{ 1, \, 2, \, \ldots, \, N' \} ,
\]
for some point $L_i\in\overline\Omega$. Some of the points $L_i$ might coincide; therefore, we relabel them as $a_1, \, a_2, \, \ldots, \, a_{N}$, with~$N\leq N'$, in such a way that $a_i \neq a_j$ if $i\neq j$. 

For the time being, we cannot exclude the possibility that $a_i\in\partial\Omega$, for some index $i$. To deal with this difficulty, we enlarge a little the domain $\Omega$ and consider a smooth, bounded domain $\Omega'\supseteq\Omega$, with the same homotopy type as $\Omega$ --- for instance, we can define $\Omega'$ as a $r$-neighborhood of $\Omega$, for $r$ small enough. 
Also, we fix a smooth function $\overline g \colon \Omega'\setminus\Omega \to \NN$, such that $\overline g = g$ on $\partial\Omega$ and 
$\norm{\nabla\overline g}_{L^2(\Omega'\setminus\Omega)} \leq C \norm{g}_{H^1(\partial\Omega)}$.
From now on, we extend systematically any function $v\colon \Omega \to \NN$ with $v = g$ on $\partial\Omega$ to a map $\overline v \colon \Omega'\to \NN$, by setting $\overline v = \overline g$ on $\Omega' \setminus \Omega$.

\subsection{An upper estimate away from singularities}

Fix a number $\rho > 0$ small enough, say,
\[
 \rho < \dist(\Omega, \, \Omega') , \qquad \rho < \frac12\min_{i \neq j} \abs{a_i - a_j}  ,
\]
so that the disks $B(a_i, \, \rho)$ are mutually disjoint and contained in $\Omega'$. The aim of the following subsection is to prove the following upper bound for energy of the minimizers, away from the singularities.

\begin{prop} \label{prop: upper estimate}
There exists a constant $C$, independent of $n$ and $\rho$, and a number $N_\rho$ such that for every $n\geq N_\rho$ we have
 \[
 \frac12 \int_{\Omega' \setminus \bigcup_i B(a_i, \, \rho)} \abs{\nabla \un}^2 \leq \kappa_* \abs{\log\rho} + C  .
 \]
\end{prop}

Before facing the proof, we fix some notations. For a fixed $i$, define $\Lambda_i$ as the set of indexes $j \in\{ 1, \, 2, \, \ldots, \, N_1\}$ such that $x_j^{\varepsilon_n} \to a_i$. For $n$ sufficiently large, we have  $B(a_i, \, \rho)\supseteq B(x_j^{\varepsilon_n}, \, \lambda_0\varepsilon_n)$ if and only if $j\in\Lambda_i$. We introduce the sets
\[
 \Omega_{i, n} := B(a_i, \, \rho) \setminus \bigcup_{j \in\Lambda_i} B(x_j^{\varepsilon_n}, \, \lambda_0\varepsilon_n)  .
\]
Recall that, by Proposition~\ref{prop: palle}, we have $\dist(\un(x), \, \NN)\leq \delta$ for all $x\in \Omega_{i, n}$. Thus, we can define $\vn, \, \dn$ by 
\[
 \vn := \pi(\left.\un\right|_{\Omega_{i, n}}), \qquad  \dn := \dist(\left.\un\right|_{\Omega_{i, n}}, \, \NN).
\]
and notice that $\vn, \, \dn\in H^1(\Omega_{i, n}, \, \NN)$. 
Denote by $\eta_{j, n}$ the free homotopy class of $\vn$, restricted to~$\partial B(x_j^{\varepsilon_n}, \, \lambda_0\varepsilon_n)$, and set
\[
  \kappa_{i, n} := \inf\Big\{ \lambda_*(\gamma)\colon \gamma\in\prod_{j\in\Lambda_i} \eta_{j, n} \Big\}  ,
\]
The continuity of $\vn$ and Lemma~\ref{lemma: extension} imply
\[
 \prod_{j = 1}^{N_1} \eta_{j, n} \cap \prod_{i = 1}^k \gamma_i \neq \emptyset  .
\]
By the definition \eqref{kappa *} of $\kappa_*$ we infer
\begin{equation}
 \kappa_* \leq \sum_{i = 1}^{N} \kappa_{i, n}  , \qquad \textrm{for all } n\in\N  .
 \label{kappa n}
\end{equation}

We can assume without loss of generality that $\kappa_{i, n} > 0$ for all $i$, $n$. Indeed, if $\kappa_{i, n} = 0$ then there is no topological obstruction to the construction of Lemma~\ref{lemma: comparison}. Arguing in a similar way, we can exhibit a comparison map $\widetilde{u}_{\varepsilon_n}$, with
$\left. \widetilde{u}_{\varepsilon_n} \right|_{\partial B(a_i, \, \rho)} = \left. \un \right|_{\partial B(a_i, \, \rho)}$, such that
\[
 E_\varepsilon(\un, B(a_i, \, \rho)) \leq E_\varepsilon(\widetilde{u}_{\varepsilon_n}, B(a_i, \, \rho)) \leq C  .
\]
Applying the Chen and Struwe's result on some small ball contained in $B(a_i, \, \rho)$, we obtain $\ee(\ue) \leq C$ on~$B(a_i, \, \rho)$. In turns, this forces
\[
 \dist^2(\un(x), \, \NN) \leq C f(\un(x)) \leq C \varepsilon_n^2 \leq \delta  ,
\]
for all $x\in B(a_i, \, \rho)$ and $n$ large enough. Therefore, no singularity is contained in $B(a_i, \, \rho)$ if $\kappa_{i, n} = 0$, and the point $a_i$ can be dropped out.

After this preliminaries, we are ready to face the proof of Proposition~\ref{prop: upper estimate}. In fact, we will give an indirect proof, based on a lower estimate for the energy near the
singularities.

\begin{lemma}
 \label{lemma: sand}
 There exists a constant $C$, independent of $n$ and $\rho$, such that for all function $v\in H^1(\Omega_{i, n}, \, \NN)$ the estimate
 \[
 \frac12 \int_{\Omega_{i, n}} \abs{\nabla v}^2 \geq \kappa_{i, n}\left(\log\frac{\rho}{\varepsilon_n} - C\right)
 \]
 holds.
\end{lemma}
\begin{proof}[Sketch of the proof]
The lemma can be established arguing exactly as in \cite[Theorem 1]{Sand} (the reader is also referred to \cite{Ch}). At first, one has to consider the case $\Omega_{i, n}$ is an annulus $B_{\rho} \setminus B_{\varepsilon}$, with $0 < \varepsilon < \rho$; then, $\kappa_{i, n}$ reduces to $\lambda_*(\eta)$, where $\eta$ is the homotopy class of $\left.v\right|_{\partial B_{\rho}}$. Assuming that $v$ is smooth, a computation in polar coordinates gives
\[
 \frac12 \int_{B_{\rho} \setminus B_{\varepsilon}} \abs{\nabla v}^2 = \frac12 \int_{\varepsilon}^{\rho} \d r \int_0^{2\pi} \d\theta \left\{ r\abs{\frac{\partial v}{\partial r}}^2
 + \frac{1}{r}\abs{\frac{\partial v}{\partial\theta}}^2 \right\} \geq \frac12 \int_{\varepsilon}^{\rho} \frac{\d r}{r} \int_0^{2\pi} \d\theta \abs{\frac{\partial v}{\partial\theta}}^2
\]
and, since the definition \eqref{lambda} of $\lambda$ implies
\[
 2\pi  \int_0^{2\pi}  \abs{\frac{\partial v}{\partial\theta}}^2 \, \d\theta \geq \lambda(\eta)^2  ,
\]
we deduce
\begin{equation} \label{sand 1}
 \frac12 \int_{B_{\rho} \setminus B_{\varepsilon}} \abs{\nabla v}^2 \geq \frac{\lambda(\eta)^2}{4\pi} \log\frac{\rho}{\varepsilon} \geq \lambda_*(\eta) \log\frac{\rho}{\varepsilon}  .
\end{equation}
Having proved the lemma in this simple case, we can repeat the same argument as \cite{Sand}, the only difference being $\kappa_{i, n}$ in place of the degree.
We exploit the property \eqref{lambda * <} instead of the triangle inequality for the degrees. Finally, since we may assume 
\[
 \dist(B(x_j^{\varepsilon_n}, \, \lambda_0\varepsilon_n), \, \partial B(a_i, \, \rho)) > \rho/2
\]
for $j\in\Lambda_i$ and $n$ large enough, we can prove the analogous of \cite[Proposition]{Sand}, which reads
\[
 \frac12 \int_{\Omega_{i, n}} \abs{\nabla v}^2 \geq \kappa_{i, n}\log\left(\frac{\rho/4}{\lambda_0\varepsilon_n}\right)  .
\]
This concludes the proof.
\end{proof}

\begin{lemma} \label{lemma: near singularities}
 There exists a constant $C$, independent of $n$ and $\rho$, and a number $N_\rho$ such that for every $n\geq N_\rho$ and every $i$ we have
 \[
 \frac12 \int_{\Omega_{i, n}} \abs{\nabla \un}^2 \geq \kappa_{i, n}\left(\log\frac{\rho}{\varepsilon_n} - C\right) - C  .
 \]
\end{lemma}

\proof
The energy of $\vn$ on $\Omega_{i, n}$ is bounded by below by Lemma~\ref{lemma: sand}; moreover, the lower bound provided by \eqref{npp} entails
\[
 \abs{\nabla \un}^2 \geq \left(1 - C\dn\right) \abs{\nabla \vn}^2  .
\]
If we knew
\begin{equation}
 \int_{\Omega_{i, n}} \dn\abs{\nabla \vn}^2  \leq C  ,
\label{n s 1}
\end{equation}
then the lemma would follow. Therefore, let us introduce the set
\[
 Y_n := \left\{x\in \Omega\colon \dist(x, \, X_{\varepsilon_n}) \leq \varepsilon_n^\alpha \right\}
\]
and split the proof of \eqref{n s 1} in two cases.

\setcounter{case}{0}
\begin{case}[Estimate out of $Y_n$] Let $x\in\Omega_{i, n}\setminus Y_n$.
Then, by Propositions~\ref{prop: palle} and~\ref{prop: chen struwe} we have
\[
 e_{\varepsilon_n}(\un)(x) \leq C_\alpha \varepsilon_n^{4\alpha} .
\]
Since $\abs{\nabla\vn} \leq C \abs{\nabla\un}$,
this entails
\[
 \int_{\Omega_{i, n}} \dn\abs{\nabla \vn}^2 \leq C \int_{\Omega_{i, n}} \dn\abs{\nabla \un}^2 \leq C \varepsilon_n^{1 - 6\alpha}
\]
which implies \eqref{n s 1} if we choose $\alpha < 1/6$.
\end{case}

\begin{case}[Estimate on $Y_n$]
We apply the H\"older inequality:
\begin{equation}
 \int_{Y_n} \dn\abs{\nabla \vn}^2 \leq C \int_{Y_n} \dn\abs{\nabla \un}^2\leq C \norm{\dn}_{L^2(Y_n)} \norm{\nabla \un}_{L^4(Y_n)}^2  .
\label{n s 2}
\end{equation}
The norm of the gradient is estimated by the Gagliardo Niremberg interpolation inequality and standard elliptic regularity results. We obtain
\[
 \norm{\nabla \un}_{L^4(Y_n)} \leq C\norm{\Delta \un}_{L^2(Y_n)}^{1/2} \norm{\un}_{L^\infty(Y_n)}^{1/2}  ,
\]
which reduces to
\begin{equation}
\norm{\nabla \un}_{L^4(Y_n)} \leq C\varepsilon_n^{-1} \norm{\nabla_u f(\un)}_{L^2(Y_n)}^{1/2}
 \label{n s 3}
\end{equation}
since $\un$ verifies the Equation \eqref{EL} and its $L^\infty$ norm is bounded by Lemma~\ref{lemma: L infty}. For a fixed $v\in \NN$, a Taylor expansion of f around the point $\pi(v)$ (see Remark~\ref{remark: f}) yields
\begin{equation} \label{n s 4}
 \abs{D f(v)} \leq M_0 \dist(v, \, \NN)  .
\end{equation}
Thus, combining the Equations \eqref{n s 2}, \eqref{n s 3} and \eqref{n s 4} with \ref{hp: df}, we infer
\[
 \int_{Y_n} \dn\abs{\nabla \vn}^2 \leq C \varepsilon_n^{-2} \norm{\dn}_{L^2(Y_n)}^2 \leq M_0 \varepsilon_n^{-2} \int_{Y_n} f(\un)  .
\]
Finally, since $Y_n$ is a finite union of balls of radius $\varepsilon^\alpha$, Lemma~\ref{lemma: local estimate} implies the desired estimate \eqref{n s 1}, for a constant depending on $\alpha$. \endproof
\end{case}

Proposition~\ref{prop: upper estimate} follows now easily from Lemmas~\ref{lemma: energy estimate} and~\ref{lemma: near singularities}, with the help of \eqref{kappa n}.
Let us point out some consequences of the previous results. For a fixed a compact set $K\subseteq \Omega' \setminus \left\{a_i\right\}_{1 \leq i \leq N}$, we know by 
Proposition~\ref{prop: upper estimate} that
\begin{equation} \label{voila}
\frac12 \int_K \abs{\nabla{\un}}^2 \leq C_K  , \qquad \int_K \dist^2(\un, \, \NN) \leq C \varepsilon_n^2
\end{equation}
at least for $n \geq N_K$. Hence, up to a renamed subsequence, by a diagonal procedure we can assume
\[
\un \to u_0 \qquad \textrm{a.e. and weakly in } H^1_{\textrm{loc}}(\Omega' \setminus \{a_1, \, \ldots, \, a_{N}\})  .
\]
Passing to the limit in the second condition of \eqref{voila}, by Fatou's lemma we deduce that
\[
u_0(x) \in \NN \qquad \textrm{for a.e. } x\in \Omega' \setminus \{a_1, \, \ldots, \, a_{N}\}  .
\]
We are now in position to prove that the points $a_i$, for $i \in \{1, \, \ldots, \, N\}$, do not belong to the boundary of $\Omega$. As a byproduct of the proof, we obtain a condition for the quantities $\kappa_{i, n}$.

\begin{prop}
 For all $i \in \{1, \, \ldots, \, N\}$, the point $a_i$ is in the interior of $\Omega$. In addition, it holds that
 \begin{equation} \label{quantizzazione}
  \sum_{i = 1}^{N} \kappa_{i, n} = \kappa_* .
 \end{equation}
\end{prop}

\begin{proof}
We adapt the proof of \cite[Lemma 3]{BBH}; the reader may see also \cite{Ch}.

Assume, by contradiction, that $a_i = 0\in\partial\Omega$ for some $i \in\{ 1, \, \ldots, \, N\}$. Then, computing in polar coordinates as we did in Lemma~\ref{lemma: sand}, we obtain the inequality
\[
 \frac12 \int_{\Omega\cap (B_{\rho} \setminus B_{\varepsilon})} \abs{\nabla u_0}^2 \geq \frac{\lambda(\eta)^2}{2\pi} (1 + o_{\rho\rightarrow 0}(1)) \log\frac{\rho}{\varepsilon}
\]
in place of \eqref{sand 1}. The factor approximately equal to $(2\pi)^{-1}$, instead of $(4\pi)^{-1}$, is due to the angular variable, which spans an interval of length 
$\pi + o_{\rho \rightarrow 0}(1)$. Arguing as in \cite[Theorem]{Sand}, we can conclude
\begin{equation} \label{border 1}
  \frac12 \int_{\Omega\setminus \cup_i B(a_i, \, \rho)} \abs{\nabla u_0}^2 \geq \left(\sum_{i = 1}^{N}\alpha_i \kappa_{i, n}\right) (1 + o_{\rho\rightarrow 0}(1)) \abs{\log\rho} 
  - C \, 
\end{equation}
for a radius $\rho > 0$ small enough, so that the balls $B(a_i, \, \rho)$ are mutually disjoint, and the coefficients $\alpha_i$ are given by
\[
 \alpha_i := \begin{cases}
             1 & \textrm{if } a_i\notin \partial\Omega \\
             2 & \textrm{if } a_i \in\partial\Omega  .
            \end{cases}
\]
On the other hand, the weak $H^1_{\textrm{loc}}$ convergence of $\un$ and Proposition~\ref{prop: upper estimate} imply
\begin{equation}\label{border 2}
 \frac12 \int_{\Omega\setminus \cup_i B(a_i, \, \rho)} \abs{\nabla u_0}^2 \leq \liminf_{n\rightarrow +\infty} E_{\varepsilon_n} (\un, \, \Omega\setminus \cup_i B(a_i, \, \rho)) \leq
 \kappa_* \abs{\log\rho} + C  .
\end{equation}
Combining \eqref{border 1} and \eqref{border 2}, dividing by $\abs{\log\rho}$ then passing to the limit as $\rho\to 0$, we deduce
\[
 \sum_{i = 1}^{N}\alpha_i \kappa_{i, n} \leq \kappa_*  .
\]
In view of the inequality \eqref{kappa n}, we have
\[
 \sum_{i = 1}^{N}\alpha_i \kappa_{i, n} = \kappa_*
\]
and, since $\kappa_{i, n} > 0$ for all $i$ and $n$, it must be $\alpha_i = 1$ for all $i$, that is, the points $a_i$ do not belong to the boundary. The equality \eqref{quantizzazione} also follows.
\end{proof}

\subsection{Proof of Theorem~\ref{th: intro-convergence}}

The proof is, essentially, a refined version of the argument we used for Proposition~\ref{prop: chen struwe}. Since the result we want to prove is local, we fix a closed disk $D\subseteq \Omega'\setminus \{a_1, \, \ldots, a_N\}$ and restrict our attention to $D$. ($D$ may intersect the boundary of $\Omega$).

By Proposition~\ref{prop: upper estimate}, and changing the radius of the disk if necessary, we can assume that
\[
\int_{\partial D} \left\{ \frac12 \abs{\nabla \un}^2 + \frac{1}{\varepsilon^2_n} f(\un) \right\} \leq C .
\]
Due to the compact inclusion $H^1(\partial D) \hookrightarrow C^0(\partial D)$, we have the uniform convergence $\un \to u_0$ on $\partial D$. We perform the same construction of Lemma~\ref{lemma: comparison}, and we obtain a sequence $\omega_{\varepsilon_n}\colon D \to \NN$ of minimizing harmonic maps, and another sequence $\widetilde\omega_{\varepsilon_n}\colon D \to \R^d$ such that 
\[
\left. \omega_{\varepsilon_n} \right|_{\partial D} = \left. \pi(\un)\right|_{\partial D}  , \qquad \left. \widetilde\omega_{\varepsilon_n} \right|_{\partial D} = \left. \un\right|_{\partial D}
\]
and
\begin{equation} \label{bound omega tilde}
E_{\varepsilon_n}(\widetilde\omega_{\varepsilon_n}; \, D) \leq \frac12 \left(1 + o_{n\to+\infty}(1) \right) \norm{\nabla \omega_{\varepsilon_n}}^2_{L^2(D)} + C\varepsilon_n
\end{equation}
(compare with \eqref{bound grad v_n}, \eqref{bound v_n}). The functions $\omega_{\varepsilon_n}$ are bounded in $H^1(D)$, since $\un$ are, hence we can apply the strong compactness result of \cite{Luck} and deduce, up to subsequences,
\begin{equation} \label{strong 1}
\omega_{\varepsilon_n} \to \omega_0 \qquad \textrm{strongly in } H^1(D) ,
\end{equation}
where $\omega_0$ is a minimizing harmonic map. Passing to the limit in the boundary condition for $\omega_{\varepsilon_n}$, we see that~$\left. \omega_0\right|_{\partial D} = \left. u_0\right|_{\partial D}$. 

As $\{\un\}_{n\in\N}$ converges weakly in $H^1(D)$, we deduce
\[ 
\frac12 \norm{\nabla u_0}_{L^2(D)}^2 \leq \frac12 \liminf_{n\to+\infty} \norm{\nabla \un }_{L^2(D)}^2
\]
but, on the other hand, \eqref{bound omega tilde} and \eqref{strong 1} give
\[
\frac12 \limsup_{n\to+\infty}  \norm{\nabla \un }_{L^2(D)}^2 \leq \limsup_{n\to+\infty} E_{\varepsilon_n} (\un; \, D) \leq \frac12 \norm{\nabla \omega_0}_{L^2(D)}^2 \leq  \frac12 \norm{\nabla u_0}_{L^2(D)}^2 .
\]
These inequalities, combined, yield
\[
\lim_{n\to+\infty}\norm{\nabla \un}_{L^2(D)} = \frac12 \norm{\nabla u_0 }_{L^2(D)}^2 = \frac12  \norm{\nabla \omega_0 }_{L^2(D)}^2.
\]
As a consequence, the convergence $\un\to u_0$ holds in $H^1(D)$ and the limit map $u_0$ is minimizing harmonic. In particular, $u_0$ solves the harmonic map equation in $D$, and the regularity theory of Morrey (see \cite{Morrey}) applies, entailing $u_0\in C^\infty(D)$. Also, as a byproduct of this argument, we obtain
\begin{equation} \label{f va a zero}
\frac{1}{\varepsilon_n^2}\int_D f(\un) \to 0 \qquad \textrm{as } n\to +\infty .
\end{equation}

Finally, we check the locally uniform convergence. Owning to the strong convergence in $H^1(D)$ and \eqref{f va a zero}, for all $\eta > 0$ we can find a radius $r > 0$, such that the inequality
\[
\int_{B(x_0, \, r)} e_{\varepsilon_n}(\un) \leq \eta
\]
holds for all $x_0\in \frac12 D$ and all $n\in\N$. Then, choosing $\eta$ small enough, we apply the Chen and Struwe's result, to infer
\[
e_{\varepsilon_n}(\un)(x_0) \leq E_{\varepsilon_n}(\un ; \, D) \leq C \qquad \textrm{for all } x\in \frac12 D  .
\]
This provides a bound for $\un$ in $W^{1, \, \infty}(D)$, which allows us to conclude the proof, by means of the Ascoli-Arzel\`a theorem.

 \section{The behavior of $u_0$ near the singularities}
 \label{sect: singularity}

 In this section, we analyze the behavior of $u_0$ near the singularities: our aim is to prove Proposition~\ref{prop: intro-geodetica}. As we already mentioned in the Introduction, we consider here just the case $\NN = \PR$. This provide a remarkable simplification in the arguments, due to the simple homotopic structure of the real projective plane, whose fundamental group consists of two elements only. Hence, there is a unique class of non homotopically trivial loops.
 
 This property reflects on the structure of the limit map. Remind that, for all $i$ and $n$, we have set
 \[
  \kappa_{i, n} = \lambda_*\left( \prod_{j\in\Lambda_i} \eta_{j, n}\right) ,
 \]
 where $\eta_{j, n}$ is the free homotopy class of $\pi(\un)$, restricted to $\partial B(x_j^{\varepsilon_n}, \, \lambda_0\varepsilon_n)$. 
 It follows from Lemma~\ref{lemma: extension} that $\prod_{j\in\Lambda_i} \eta_{j, n}$ is the homotopy class of $\pi(\un)$ restricted to $\partial B(a_i, \, \rho)$, for a small radius $\rho > 0$. Since the homotopy class is stable by uniform convergence, from Theorem~\ref{th: intro-convergence} we deduce
 \[
  \kappa_{i, n} = \lambda_* \left(\textrm{homotopy class of }  \left. u_0 \right|_{\partial B(a_i, \, \rho)}\right)  ,
 \]
 that is, $\kappa_{i, n}$ is independent of $n$.
 On the other hand, there is a unique non zero value that $\kappa_{i, n}$ and $\kappa_*$ can assume, corresponding to the unique class of non trivial loops. As a consequence, from \eqref{quantizzazione} we infer that that there is at most one index $i$ such that $\kappa_{i, n} \neq 0$, and we prove the following
 
 \begin{lemma}
  In case $\NN \simeq \PR$, there exists a point $a\in\Omega$ such that $u_0\in C^\infty(\Omega\setminus\{a\})$.
 \end{lemma}
 
 Assume now that the boundary datum is non homotopically trivial. Up to a translation, we can suppose that the unique singular point of $u_0$ is the origin, and we fix a radius $r > 0$ such that $B(0, \, r) \subseteq \Omega$. We also introduce the functions $R, S\in C^\infty(0, \, r)$ by
 \[
  R(\rho) := \frac12 \int_{\partial B_\rho} \abs{\frac{\partial u_0}{\partial\nu}}^2 \, \d\mathcal H^1 
 \]
 and
 \[
  S(\rho) := \frac{1}{2\rho} \int_{\partial B_\rho} \abs{\nabla_T u_0}^2 \, \d\mathcal H^1 
 \]
 where $\nabla_T$ denotes the tangential derivation. These functions are obviously non negative; in fact, $S$ is bounded by below by $\kappa_*$. Indeed, by definition of $\lambda$ we have for all $\rho\in (0, \, r)$
 \[
  4\pi S(\rho) = 2\pi \int_0^{2\pi} \abs{c_\rho'(\theta)}^2  \, \d\theta \geq \lambda^2(\gamma)  ,
 \]
 where $c_\rho$ is the function considered in Proposition~\ref{th: intro-convergence}, and
 \[
  S(\rho) \geq \frac{\lambda^2(\gamma)}{4\pi} = \kappa_*  .
 \]

 \begin{lemma} \label{lemma: singularity 1}
  The function $\rho \mapsto \rho^{-1}(S(\rho) - \kappa_*)$ is summable over $(0, \, r)$. In particular,
  \[
   \liminf_{\rho\to 0}S(\rho) = \kappa_*  .
  \]
 \end{lemma}
 \proof
 Let $0 < \rho_0 < \min\{r, \, 1\}$. With the help of Theorem~\ref{th: intro-convergence}, we can pass to the limit as $n\to +\infty$ in Proposition~\ref{prop: upper estimate}, to obtain
 \[
  \frac12 \int_{B_r(0) \setminus B_{\rho_0}(0)} \abs{\nabla u_0}^2 \leq \kappa_* \abs{\log\rho_0} + C 
 \]
 and, expressing the left-hand side in polar coordinates,
 \begin{equation} \label{singularity 1, 1}
  \int_{\rho_0}^r \left\{ R(\rho) + \frac{1}{\rho} S(\rho) \right\} \, \d\rho \leq \kappa_* \abs{\log\rho_0} + C  .
 \end{equation}
 Taking advantage of this bound, we compute
 \[
  \int_{\rho_0}^r \frac{1}{\rho}\left( S(\rho) - \kappa_* \right) \, \d\rho = \int_{\rho_0}^r \frac{S(\rho)}{\rho} \, \d\rho - \kappa_*\abs{\log\rho_0} - \kappa_* \log r \leq C  .
 \]
 Letting $\rho_0 \searrow 0$, we deduce the summability of $\rho \mapsto \rho^{-1}(S(\rho) - \kappa_*)$ which, in turns, forces the inferior limit of $S - \kappa_*$ to vanish.
 \endproof
 
 Proposition~\ref{prop: intro-geodetica} follows easily from this lemma. Indeed, we can pick a sequence $\rho_n\searrow 0$ such that $S(\rho_n)\to \kappa_*$: this is a minimizing sequence for the length-squared functional
 \[
  c \in H^1(S^1, \, \NN) \mapsto \frac{1}{2} \int_0^{2\pi} \abs{c'(\theta)}^2 \, \d\theta
 \]
 under the constraint that $c$ is not homotopically trivial, and hence, by the compact inclusion $H^1(S^1, \, \NN)\hookrightarrow C^0(S^1, \, \NN)$, it admits a subsequence uniformly converging to a minimizer, which is a geodesic. The continuous inclusion $H^1(S^1, \, \NN)\hookrightarrow C^{1/2}(S^1, \, \NN)$ and interpolation in H\"older spaces provide also the convergence in $C^\alpha$, for $\alpha \in (0, \, 1/2)$.
 
 We are not able to say whether the convergence holds for the whole family $\{c_\rho\}_{\rho > 0}$, because we are not able to identify the limit geodesic $c_\rho$. However, we state here some additional properties we have been able to prove about the functions $S$ and $R$, in the hope that they might be of interest for future work.
 
 \begin{lemma} \label{lemma: singularity 2}
  It holds that 
  \[
   R(\rho) = \frac{1}{\rho} \left( S(\rho) - \kappa_* \right)  .
  \]
 \end{lemma}
 \proof
 We claim that 
 \begin{equation} \label{singularity 2, 1}
  \frac{\d}{\d\rho} \left( \rho R(\rho) - S(\rho)\right) = 0  .
 \end{equation}
 This equality is essentially a consequence of the Pohozaev identity for the harmonic maps, but here we will present its proof in a slightly different form.
 Since $u_0$ is harmonic away from $0$, its Laplacian $\Delta u_0$ is, at every point, a normal vector to $\NN$. Thus, for each point $x\in \Omega\setminus \{0\}$ we have
 \[
 \left(\Delta u_0 \cdot \frac{\partial u_0}{\partial\nu}\right) (x) = 0  ,
 \]
 where $\nu = x/\abs{x}$. We multiply the previous identity by $\abs{x}^2$, pass to polar coordinates, and integrate with respect to $\theta\in[0, \, 2\pi]$, for a fixed $\rho\in [0, \, r]$. This yields
 \[
 \rho^2 \int_0^{2\pi} \frac{\partial^2 u_0}{\partial\rho^2}\cdot \frac{\partial u_0}{\partial \rho} \, \d\theta + \rho \int_0^{2\pi} \abs{\frac{\partial u_0}{\partial\rho}}^2 \, \d\theta +
 \int_0^{2\pi} \frac{\partial^2 u_0}{\partial \theta^2} \cdot \frac{\partial u_0}{\partial\rho} \, \d\theta = 0
 \]
 and, after an integration by parts in the third term,
 \[
  \frac{\rho^2}{2} \frac{\d}{\d\rho} \int_0^{2\pi} \abs{\frac{\partial u_0}{\partial \rho}}^2 \, \d\theta +  \rho \int_0^{2\pi} \abs{\frac{\partial u_0}{\partial\rho}}^2 \, \d\theta -
  \frac12 \frac{\d}{\d\rho} \int_0^{2\pi} \abs{\frac{\partial u_0}{\partial\theta}}^2 \, \d\theta = 0  .
 \]
 This equality can be rewritten as
 \[
  \rho^2 \frac{\d}{\d\rho} \left( \frac{R(\rho)}{\rho} \right) + 2R(\rho) - \frac{\d}{\d\rho} S(\rho) = 0  ,
 \]
 from which we deduce \eqref{singularity 2, 1}. Our claim is proved.

 As a consequence of \eqref{singularity 2, 1}, there exists a constant $k$ such that
 \[
  R(\rho) = \frac{1}{\rho} \left(S(\rho) + k\right)  ,
 \]
 and the lemma will be proved once we have identified the value of $k$. To do so, fix $0 < \rho_0 < \min\{r, \, 1\}$ and notice that \eqref{singularity 1, 1} implies
 \[
  \kappa_* \abs{\log\rho_0} + C \geq \int_{\rho_0}^r \frac{1}{\rho} \left( 2S(\rho) + k \right) \, \d\rho = \int_{\rho_0}^r \frac{2}{\rho} \left( S(\rho) - \kappa_* \right) \, \d\rho + \int_{\rho_0}^r \frac{1}{\rho} (2\kappa_* + k) \, \d\rho
 \]
 The first integral at the right-hand side is non negative, since $S \geq \kappa_*$. Therefore, for small values of $\rho_0$,
 \[
 \kappa_* \abs{\log\rho_0} + C \geq (2\kappa_* + k)\abs{\log\rho_0} - C
 \]
 and, comparing the coefficients of the leading terms, we have $\kappa_* \geq 2\kappa_* + k$, that is, $k \leq - \kappa_*$. On the other hand,
 \[
  0 \leq \rho R(\rho) = S(\rho) + k
 \]
 and, taking the inferior limit as $\rho \searrow 0$, by Lemma~\ref{lemma: singularity 1} we infer $0 \leq \kappa_* + k$, which provides the opposite inequality $k \geq -\kappa_*$.
 \endproof
 
 Lemmas~\ref{lemma: singularity 1} and~\ref{lemma: singularity 2} combined imply that $R \in L^1(0, \, r)$. 
 
 \begin{remark}
  If we knew that $R$ has better integrability properties, for instance $R\in L^p$ for some $p > 1$ (or even $R^{1/2}\in L^{(2, \, 1)}$), then we could conclude the convergence of the whole family $\{c_\rho\}_{\rho > 0}$, at least in $L^1(S^1, \, \NN)$. Indeed, applying the fundamental theorem of calculus, the Fubini-Tonelli theorem, and the H\"older inequality, we would obtain
  \[
  \norm{c_{\rho_1} - c_{\rho_2}}_{L^1(S^1)} \leq \int_0^{2\pi} \d\theta \int_{\rho_1}^{\rho_2} \d\rho \abs{\frac{\partial u_0}{\partial\rho}} \leq \int_{\rho_1}^{\rho_2} \d\rho (2\pi\rho)^{1/2} \left\{\int_0^{2\pi} \d\theta \abs{\frac{\partial u_0}{\partial\rho}}^2 \right\}^{1/2}
  \]
  and hence
  \[
   \norm{c_{\rho_1} - c_{\rho_2}}_{L^1(S^1)} \leq \int_{\rho_1}^{\rho_2} \left(\frac{4\pi R(\rho)}{\rho}\right)^{1/2} \, \d\rho  ,
  \]
  where the right-hand side converges to zero as $\rho_1, \, \rho_2 \to 0$, again by the H\"older inequality. Thus, $\{c_\rho\}_{\rho > 0}$ would be a Cauchy sequence in $L^1(S^1, \, \NN)$.
 \end{remark}

\textbf{Acknowledgements.} The author would like to thank professor Fabrice Bethuel for bringing the problem to his attention and numerous helpful discussions, as well as professor Arghir Zarnescu for his interesting remarks and Jean-Paul Daniel for his help in proofreading.

\bibliographystyle{acm}
\bibliography{amsart_biaxiality_2d}

\end{document}